\documentclass[11pt,letterpaper]{amsart}
\pdfoutput=1
\usepackage{geometry}                
\usepackage{graphicx}
\usepackage{amsfonts,amsthm,amsmath,amssymb,latexsym, amscd, euscript}
\usepackage{graphicx,color}
\usepackage[all]{xy}
\usepackage{epsfig}
\usepackage{labelfig}
\usepackage{mathrsfs}

\usepackage{amssymb}
\usepackage{epstopdf}
\theoremstyle{plain}
\newtheorem{thm}{Theorem}[section]
\newtheorem{cor}[thm]{Corollary}
\newtheorem{prop}[thm]{Proposition}
\newtheorem{lem}[thm]{Lemma}

\newtheorem{definition}[thm]{Definition}

\title[Quadratic differentials]{Quadratic differentials and foliations on infinite Riemann surfaces}

\thanks{The author was partially supported by the Simons Foundation Collaboration Grant, (346391) and by PSCCUNY grants (61582 and 63477).}

\author{Dragomir \v Sari\' c}

\begin{document}


\begin{abstract} 
We prove that an infinite Riemann surface $X$ is parabolic ($X\in O_G$) if and only if the union of the horizontal trajectories of any integrable holomorphic quadratic differential that are cross-cuts is of zero measure. Then we establish the density of the Jenkins-Strebel differentials in the space of all integrable quadratic differentials when $X\in O_G$ and extend Kerckhoff's formula for the Teichm\"uller metric in this case. Our methods depend on extending to infinite surfaces the Hubbard-Masur theorem describing which measured foliations can be realized by horizontal trajectories of integrable holomorphic quadratic differentials.
 \end{abstract}

\maketitle

\section{Introduction}

A Riemann surface $X$ is called {\it infinite} if its fundamental group is not finitely generated. A  classification of infinite Riemann surfaces from the point of view of the function theory is into parabolic and non-parabolic surfaces \cite{AhlforsSario}. A Riemann surface $X$ is {\it parabolic}, in notation $X\in O_G$, if it does not have a Green's function, i.e., a harmonic function with a logarithmic singularity at a single point that is $0$ at infinity. It turns out that $X\in O_G$ is equivalent to the Brownian motion being recurrent (see Ahlfors-Sario \cite{AhlforsSario}), which is also equivalent to the (hyperbolic) geodesic flow on the unit tangent bundle of $X$ being ergodic (see Tsuji \cite{Tsuji}, Sullivan \cite{Sullivan} and Nicholls \cite{Nicholls1}) which is equivalent to the divergence of its Poincar\'e series (see  \cite{Nicholls1}) which is equivalent to the covering Fuchsian group having Bowen's property (see Astala-Zinsmeister \cite{Astala-Zinsmeister} and Bishop \cite{Bishop}). 

It is common to associate a holomorphic quadratic differential with a Riemann surface. When $X$ is an infinite Riemann surface, the space of integrable holomorphic quadratic differentials $A(X)$ plays an important role. Reich and Strebel \cite{ReichStrebel} proved that affine stretching in the natural parameter of an integrable holomorphic quadratic differential on $X$ produces a map with the smallest quasiconformal constant in its homotopy class (which generalizes the original result of Teichm\"uller for closed surfaces). Even a more direct relationship between $X$ and the structure of the space $A(X)$ is given by Markovic \cite{markovic} when he showed that a linear isometry between $A(X)$ and $A(Y)$ is geometric (i.e., Riemann surfaces $X$ and $Y$ are quasiconformal) thus extending Royden's theorem to all infinite surfaces.

We prove that $X\in O_G$ if and only if, for each $\varphi\in A(X)$, the set of horizontal trajectories of $\varphi$ that are cross-cuts is of zero area. The above condition is required to hold for all $\varphi\in A(X)$, and it is equivalent to the ergodicity of the geodesic flow on the unit tangent bundle of $X$. When $X\in O_G$, we establish that each $\varphi\in A(X)$ is approximated by a sequence of  Jenkins-Strebel quadratic differentials with a single cylinder on $X$ thus extending a result of Masur \cite{Masur}. Further, we establish that for $X\in O_G$, the Teichm\"uller distance between two points is the supremum of one-half of the logarithm of the quotient of extremal lengths of the simple closed curves on the two Riemann surfaces representing two points which is an extension of Kerckhoff's result for closed surfaces \cite{Kerckhoff}. Given the above results, we would like to understand which surfaces are parabolic. 
 In the case of a Cantor tree surface $X_C$, we complement the results of McMullen \cite{mcmullen} and Basmajian, Hakobyan, and the author \cite{BHS} to provide conditions on the lengths of cuffs of a pants decomposition to decide whether $X_C\in O_G$ or  $X_C\notin O_G$. The methods in the paper depend on the ability to construct ``convenient'' integrable holomorphic quadratic differentials, and we establish this by extending the Hubbard-Masur characterization \cite{HubbardMasur}  
of measured foliations that can be realized by quadratic differentials. More details follow.

As mentioned above, many authors (see \cite{Nevanlinna:criterion,AhlforsSario,Nicholls1,Sullivan,Astala-Zinsmeister,Bishop,Fernandez-Melian}) have established equivalent characterizations to the parabolicity of Riemann surfaces. We add one more equivalent property (see Theorem \ref{thm:par-ch-cross-cuts}). 
A subset of $X$ is of {\it zero measure} if it is a countable union of subsets mapped to zero measure sets in the charts. 

\begin{thm}
The Brownian motion on an arbitrary hyperbolic Riemann surface $X$ is recurrent if and only if a.e. horizontal leaf of every finite area holomorphic quadratic differentials on $X$ is recurrent.
\end{thm}

A differential $\varphi\in A(X)$ is said to be {\it Jenkins-Strebel} if all non-singular horizontal trajectories are closed, and they make a single cylinder. Masur \cite{Masur} proved that Jenkins-Strebel differentials are dense for compact surfaces for the topology of the pointwise convergence. For infinite Riemann surfaces, the $L^1$-topology is stronger, and we prove 
 (see Corollary \ref{cor:jenkins-dense-L1})

\begin{thm}
Let $X\in O_G$ be an infinite Riemann surface and $\varphi\in A(X)$. Then there exists a sequence $\{\varphi_n\}_n$ of Jenkins-Strebel differentials on $X$ such that
$$
\int_X|\varphi -\varphi_n|\to 0
$$
as $n\to\infty$.
\end{thm}

Using the above theorem and an approximation of points in the Teichm\"uller space $T(X)$ by Strebel points (see \cite[page 106, Theorem 12]{GardinerLakic}) we extend Kerckhoff's formula \cite{Kerckhoff} for the Teichm\"uller metric in terms of the extremal lengths of simple closed curves from compact surfaces to parabolic surfaces (see Theorem \ref{thm:kerckhoff}).

\begin{thm}
Let $X\in O_G$ and $S$ be the set of simple closed geodesics on $X$. The Teichm\"uller distance between two points $[f:X\to Y]$ and $[g:X\to Z]$ in $T(X)$ is equal to
$$
d_T([f],[g])=\frac{1}{2}\log \sup_{\gamma\in S}\frac{\mathrm{ext}_Z(g(\gamma ))}{\mathrm{ext}_Y(f(\gamma ))}
$$
where $\mathrm{ext}_Y(\cdot )$ and $\mathrm{ext}_Z(\cdot )$ are extremal lengths on surfaces $Y$ and $Z$ of the corresponding simple closed curves.
\end{thm}

The above results underscore the importance of classifying surfaces into parabolic and non-parabolic surfaces. We use our methods to improve the existing classification in terms of the lengths of cuffs of geodesic pants decomposition of the surfaces.

Consider the Cantor tree surface $X_C$ and a geodesic pants decomposition as in Figure 7. At the level $n$, there are $2^n$ pairs of pants, and the union of all pairs of pants up to and including level $n$ pants is a compact subsurface $X_n$ with $2^{n+1}$ geodesic boundaries. Assume that the length of each boundary component of $X_n$ is equal to $\ell_{n+1}$. Basmajian, Hakobyan, and the author \cite{BHS} proved that if
$$
\ell_n\leq \frac{n}{2^n}
$$
then $X\in O_G$. McMullen \cite{mcmullen} proved that if
$$
\ell_n\geq C>0
$$
for all $n$ then $X\notin O_G$. Given that $X$ has many directions to escape towards its ends, it would be reasonable to expect that some choice of $\ell_n$ converging to zero would also give a non-parabolic surface. We show how small $\ell_n$ is allowed so that $X_C$ is not parabolic (see Theorem \ref{thm:cantor}).

\begin{thm}
Let $X_C$ be the Cantor tree surface with geodesic pants decomposition such that the level $n$ boundary geodesics lengths equal some $\ell_n$ for each $n$. If there exists $r>2$ such that
$$
\ell_n\geq \frac{n^r}{2^n}
$$
then $X\notin O_G$.
\end{thm}

A classical result of Mori \cite{Mori} and (in a more general form of) Rees \cite{Rees} states that $\mathbb{Z}^s$-covers of a compact Riemann surface are parabolic for $s=1, 2$ and not parabolic for $s\geq 3$. We introduce a topological complexity function $q:\mathbb{N}\to\mathbb{R}$ associated with a finite-area subsurface decomposition $\{ X_n\}_n$ of $X$ by consecutively adding pairs of pants. The quantity $q(n)$ is the number of boundary cuffs of $X_n$ that bound a component of $X\setminus X_n$ with non-simple ends. We prove (see Theorem \ref{thm:bdd-parbolic})

\begin{thm}
Let $X$ be an infinite Riemann surface with bounded geodesic pants decomposition. If
$$
\sum_n\frac{1}{q(n)}=\infty
$$
then $X\in O_G$.
\end{thm}

The above theorem recovers the result of Mori \cite{Mori} when the Riemann surface is obtained by taking $\mathbb{Z}^2$-translates of a finite area hyperbolic Riemann surface with four boundary geodesics (for the construction, see \cite[Section 10.5]{BHS}) because then we have $q(n)\leq const\cdot n$. In fact, $X\in O_G$ when $q(n)\leq const\cdot n(\log n)^p$ for $p<1$ which is an approximation statement between $\mathbb{Z}^2$- and $\mathbb{Z}^3$-covers.

To prove the above results, 
we establish an analogous statement to the Hubbard-Masur theorem \cite{HubbardMasur} for $X=\mathbb{H}/\Gamma$ with $\Gamma$ of the first kind by describing which measured foliations arise from horizontal foliations of differentials in $A(X)$.
Each non-singular horizontal trajectory of an integrable holomorphic quadratic differential $\varphi\in A (X)$ is homotopic to a simple hyperbolic geodesic on $X$. By the push-forward of the transverse measure $\int |Im(\sqrt{\varphi})|$ to the horizontal trajectories, we obtain a map $\varphi\mapsto \mu_{\varphi}$ from $A(X)$ to the space of measured lamination $ML(X)$. By the Heights Theorem \cite{Saric-heights}, we have that the straightening map $\varphi\mapsto \mu_{\varphi}$ is injective.
We need to characterize the image of $A(X)$. 

A partial measured foliation $\mathcal{F}$ on $X$ is a collection of closed Jordan domains $\{ U_i\}_i$ and measurable sets $\{ E_i\subset U_i\}_i$ of $X$ which carry the leaves of $\mathcal{F}$ and differentiable map $v_i:U_i\to\mathbb{R}$ such that $v_i=\pm v_j+const$ on $U_i\cap U_j$. The collection $\{ U_i\}_i$ does not necessarily cover $X$; hence we call it ``partial'' foliation. The leaves of $\mathcal{F}$ are given by $v_i^{-1}(a)$ for $a\in\mathbb{R}$ and the Dirichlet integral is well-defined on compact subsets of $X$(see \S \ref{sec:foliations}). A partial foliation $\mathcal{F}$ is called integrable if all but countably many leaves are homotopic to a geodesic of $X$ and the Dirichlet integral $D(\mathcal{F})$ is finite over all $X$. 

Let $ML_{\mathrm{int}}(X)$ be the space of all measured lamination $\mu\in ML(X)$ such that there exists an integrable partial foliation $\mathcal{F}$ on $X$ whose straigthening gives measured lamination $\mu_{\mathcal{F}}=\mu$. 
We prove (see Theorems \ref{thm:main})

\begin{thm}
\label{thm:mlf}
Let $X=\mathbb{H}/\Gamma$, with $\Gamma$ a Fuchsian group of the first kind. The straightening map
$$
A(X)\to ML(X)
$$
obtained by mapping $\varphi$ to measured lamination $\mu_{\varphi}$ is a homeomorphism onto the space $ML_{\mathrm{int}}(X)$.
\end{thm}

The above theorem completely characterizes which measured laminations can arise from the horizontal foliations of integrable holomorphic quadratic differentials on a Riemann surface $X$. The space $ML_{\mathrm{int}}(X)$ is invariant under quasiconformal maps (see Theorem \ref{thm:mlf-qc-inv}), but it is not invariant under arbitrary homeomorphisms. In other words, $ML_{\mathrm{int}}(X)$ varies as different Riemann surface structures are placed on a single topological surface of an infinite type, unlike in the case of surfaces of finite type.
This phenomenon underscores the difference between Teichm\"uller spaces of finite and infinite surfaces.

\vskip .1 cm

\noindent {\it Acknowledgements.} We thank Curtis McMullen for pointing out that Theorem 1.1 holds for arbitrary Riemann surfaces. We are grateful to anonymous referees for greatly improving the paper's clarity.

\section{Integrable holomorphic quadratic differentials and partial measured foliations}
\label{sec:foliations}

A {\it quadratic differential} $\varphi$ on a Riemann surface $X$ is an assignment of a function $\varphi (z)$ in each coordinate chart $z$ such that $\varphi (z)dz^2$ remains invariant under the coordinate change. A quadratic differential $\varphi$ is {\it holomorphic} if $\varphi (z)$ is a holomorphic function in each coordinate chart of $X$. A quadratic differential $\varphi$ is {\it integrable} if $\int_X|\varphi |<\infty$, where the absolute value $|\varphi (z)|$ is the area form.

A Riemann surface $X$ is said to be {\it hyperbolic} if it contains a unique metric of curvature $-1$ in its conformal class. A hyperbolic Riemann surface is conformally equivalent to $\mathbb{H}/\Gamma$, where $\mathbb{H}$ is the upper half-plane and $\Gamma$ is a Fuchsian group acting on $\mathbb{H}$. The hyperbolic metric on $X$ is given by the projection of the hyperbolic metric on $\mathbb{H}$. 

A holomorphic quadratic differential $\varphi$ on a hyperbolic Riemann surface $X$ defines a natural parameter away from its zeroes as follows. Consider a simply-connected local chart $z$ of a neighborhood of a point $P\in X$, where $z_0$ corresponds to $P$ and $\varphi (z)\neq 0$ in the whole chart. Then $\sqrt{\varphi (z)}$ is a well-defined holomorphic function in the entire chart, and the {\it natural parameter} is given by
$$
w(z)=u+iv=\int_{z_0}^z\sqrt{\varphi (z)}dz .
$$
Different choices of the local chart $z$, the base point $P$, and the square root give a different natural parameter $w_1$. However, we have $w_1=\pm w+const$ on the intersection of the charts.  Therefore the horizontal and vertical lines in the $w$-parameter are mapped onto the horizontal and vertical lines in the $w_1$-parameter. A {\it horizontal arc} of $\varphi$ is an arc on $X$ that is the preimage of a horizontal line in the natural parameter of $\varphi$. A {\it horizontal trajectory} of $\varphi$ is a maximal horizontal arc. A horizontal trajectory can either be closed or open. If it is open, it can accumulate to a zero of $\varphi$ in either direction. Analogously, a {\it vertical arc} is the preimage of a vertical line in the natural parameter, and a {\it vertical trajectory} is a maximal vertical arc.

The {\it horizontal foliation} of $\varphi$ on $X$ is a foliation whose leaves are horizontal trajectories. The transverse measure is given by $\int_{\alpha}|Im(\sqrt{\varphi (z)}dz)|=\int_{\alpha}|dv|$, where $\alpha$ is a differentiable arc transverse to the horizontal leaves. At a zero of $\varphi$ of order $k>0$, the horizontal foliation is strictly speaking not a foliation but instead has a well-known structure of $(k+2)$-pronged singularity (see \cite{Strebel}). 
 
A Riemann surface $X=\mathbb{H}/\Gamma$ is said to be {\it infinite} if $\Gamma$ is not finitely generated. 
In this paper, we study integrable holomorphic quadratic differentials $\varphi$ on an infinite hyperbolic Riemann surface $X$.

It will be useful for our purposes to have a notion of a partial measured foliation whose definition is motivated by the paper of Gardiner and Lakic (see \cite{GardinerLakic}) where they considered closed surfaces. We are mainly interested in hyperbolic surfaces with infinite area. Even when $X$ has a finite hyperbolic area, our definition is slightly more general than the definition in \cite{GardinerLakic}.

\begin{definition}
\label{def:partial-fol}
Consider $X=\mathbb{H}/\Gamma$ as an infinite Riemann surface such that $\Gamma$ is a  Fuchsian group of the first kind.
A {\it partial measured foliation} $\mathcal{F}$ on $X$ consists of countably many triples $\{ (v_i, U_i,E_i)\}_i$ where $U_i\subset X$ is a closed Jordan domain, $E_i\subset U_i$ is a measurable set and $v_i:U_i\to\mathbb{R}$ is a continuous function satisfying the following conditions:
\begin{enumerate}
\item[(i)] The family $\{ U_i\}_i$ is locally finite on $X$.

\item[(ii)]
The boundary of $U_i$, denoted as $\partial U_i$, is piecewise smooth and is divided into four closed arcs with non-overlapping interiors. Within this partition, two specific opposite arcs, namely $a_i^1$ and $a_i^2$, are chosen as the vertical sides of $\partial U_i$.

\item[(iii)]
The functions $v_i:U_i\to\mathbb{R}$ are differentiable with surjective $dv_i$ at the tangent space of each point of the interior of $U_i$. 

By the Implicit Function Theorem, the pre-image $v_i^{-1}(c)$ for $c\in\mathbb{R}$, if non-empty, is an open differentiable arc. We require that $v_i^{-1}(c)$ at one end accumulates to a unique point in $a_i^1\subset \partial U_i$ and the other end to a unique point in $a_i^2\subset \partial U_i$ and that $E_i$ is foliated by the differentiable arcs  $v_i^{-1}(c)$ for $c\in\mathbb{R}$.

\item[(iv)] For any two sets $U_i$ and $U_j$,  we have
 \begin{equation}
\label{eq:coord-change}
v_i=\pm v_j+const
\end{equation} 
on $U_i\cap U_j$.

\item[(v)] For any two $E_i$ and $E_j$, we have
$$
E_i\cap \bar{U}_j=E_j\cap\bar{U}_i
$$
where $\bar{U}$ is the closure of $U$. 
\end{enumerate}
\end{definition}

Since $dv_i$ are surjective, the Implicit Function Theorem implies that the sets $U_i$ are foliated by differentiable arcs $v_i^{-1}(c)$ for $c\in\mathbb{R}$. Therefore $U_i$ has a product structure with leaves being horizontal arcs.

The condition {\rm (\ref{eq:coord-change})} is the standard condition as in the case of the transition maps for the natural parameters of a holomorphic quadratic differential. Two sets $U_i$ and $U_j$ can intersect in a set with an interior or only along their boundaries. When  $U_i$ and $U_j$ intersect along their vertical sides, the condition {\rm (\ref{eq:coord-change})} imposes the equality of the transverse measure on the common intersection. The functions $v_i$ and $\pm v_j+const$  glue to a piecewise differentiable function in a neighborhood of the vertical boundary sides common to $U_i$ and $U_j$. 

The sets $E_i\subset U_i$ comprise of the local leaves of the partial foliation $\mathcal{F}$ and are required to be measurable in order for us to be able to apply the integration. The condition (v) implies that the leaves can be extended across the intersecting charts $U_i\cap U_j$. 
This flexibility of ``gluing only along the vertical sides and measurable sets''  allows us to construct important examples of partial foliations in subsequent sections. 

\begin{definition}
\label{def:h-trajectory}
For a partial measured foliation $\mathcal{F}$, a {\it horizontal arc} is a curve in $X$ that is a connected union of $v_i^{-1}(c_i)\subset E_i$ for some finite or infinite choice of indices $i$ and real numbers $c_i$. A {\it horizontal trajectory} of $\mathcal{F}$ is a maximal horizontal arc, including the possibility of a closed trajectory.
\end{definition}

A set of horizontal arcs of $\mathcal{F}$ is called a {\it measurable strip} if it is homeomorphic to a product of a measurable set and horizontal arcs. The sets $E_i$ are measurable strips.

Given a closed curve $\gamma$ on $X$, the {\it height} of (the homotopy class) of $\gamma$ with respect to a partial foliation $\mathcal{F}$ is given by
$$
h_{\mathcal{F}}(\gamma )=\inf_{\gamma_1} \int_{\gamma_1\cap\mathcal{F}}|dv|
$$
where the infimum is over all closed curves $\gamma_1$ in $X$ homotopic to $\gamma$ and $\gamma_1\cap\mathcal{F}$ is the union of $\gamma_1\cap E_i$ over all indices $i$. The integration of the form $|dv|$ is independent of the chart by (\ref{eq:coord-change}) and if a subarc $\gamma_2$ of $\gamma_1$  does not intersect the support of $\mathcal{F}$ we set $\int_{\gamma_2}|dv|=0$. Since we assumed that the functions $v_i:U_i\to\mathbb{R}$ are $C^1$, it follows that $\int_{\gamma\cap\mathcal{F}}|dv|<\infty$ for each compact differentiable curve or arc $\gamma$ which implies $h_{\mathcal{F}}(\gamma )<\infty$ for every simple closed curve $\gamma$. 

If $Y$ is a subsurface of $X$ and $\gamma\subset Y$, then we define
$$
h_{\mathcal{F},Y}(\gamma )=\inf_{\gamma_1}\int_{\gamma_1\cap\mathcal{F}}|dv|
$$
where the infimum is over all closed curves $\gamma_1$ in $Y$ that are homotopic to $\gamma$. 

\begin{definition}
A horizontal trajectory of $\mathcal{F}$ that limits to a point in $X$ is called a singular trajectory.
A partial measured foliation $\mathcal{F}$ on $X$ is said to be proper if, except countably many singular trajectories, the lift to the universal cover $\mathbb{H}$ of each horizontal trajectory accumulates to an ideal endpoint on $\partial\mathbb{H}$ on each end and the two points of the accumulation are distinct. \end{definition}

If $\mathcal{F}$ is a proper partial foliation of $X$, then the lift $\tilde{\mathcal{F}}$ to $\mathbb{H}$ is a proper partial measured foliation of $\mathbb{H}$. Its set of horizontal trajectories is invariant under $\Gamma$.
Each non-singular horizontal trajectory of $\tilde{\mathcal{F}}$ has exactly two ideal endpoints on the ideal boundary $\partial\mathbb{H}$ of $\mathbb{H}$. It can be pulled tight to a hyperbolic geodesic of $\mathbb{H}$ with the same endpoints on $\partial\mathbb{H}$. 
We denote by $G(\tilde{\mathcal{F}})$ the closure of the set of geodesics obtained by replacing the non-singular horizontal trajectories of $\tilde{\mathcal{F}}$ with the hyperbolic geodesics that share the same endpoints on $\mathbb{H}$. 
To see that $G(\tilde{\mathcal{F}})$ is a geodesic lamination in the hyperbolic plane $\mathbb{H}$, it is enough to establish that any two geodesics obtained by pulling tight two non-singular horizontal trajectories of $\tilde{\mathcal{F}}$ do not intersect. If they do intersect, then the pairs of their endpoints separate each other on $\partial\mathbb{H}$ and the non-singular trajectories of $\tilde{\mathcal{F}}$ limiting to these pairs intersect in $\mathbb{H}$ which is a contradiction. Finally, the closure of a set of non-intersecting geodesics in $\mathbb{H}$ consists of non-intersecting geodesics; therefore, $G(\tilde{\mathcal{F}})$ is a geodesic lamination.

By repeating the arguments in \cite[\S 3.2]{Saric20}, the transverse measure to $\mathcal{F}$  induces a transverse measure to $\tilde{\mathcal{F}}$ which in turn induces a transverse measure on $G(\tilde{\mathcal{F}})$. Denote by $\mu_{\tilde{\mathcal{F}}}$ the induced measured lamination in $\mathbb{H}$ and note that it is invariant under the action of $\Gamma$. Therefore the measured lamination $\mu_{\tilde{\mathcal{F}}}$ induces a measured lamination $\mu_{{\mathcal{F}}}$ on $X$ and we will call this process {\it straightening} of $\mathcal{F}$.

\section{Realizing integrable foliations by quadratic differentials}

Let $\mathcal{F}$ be a proper partial measured lamination on $X$ given by differentiable maps $v_i:U_i\to\mathbb{R}$ on a collection of pairs of sets $\{ (U_i, E_i)\}_i$. The collection  $\{ U_i\}_i$ is locally finite in $X$, i.e., every compact set in $X$ intersects only finitely many sets of the collection. 

The Dirichlet integral of $v_i$ on $E_i\subset U_i$ is given by $\int_{E_i}[(\frac{\partial v_i}{\partial x})^2+ (\frac{\partial v_i}{\partial y})^2]dxdy$ and by (\ref{eq:coord-change}) we have
$$
\int_{E_i\cap E_j}\Big{[}\Big{(}\frac{\partial v_i}{\partial x}\Big{)}^2+ \Big{(}\frac{\partial v_i}{\partial y}\Big{)}^2\Big{]}dxdy=\int_{E_i\cap E_j}\Big{[}\Big{(}\frac{\partial v_j}{\partial x}\Big{)}^2+ \Big{(}\frac{\partial v_j}{\partial y}\Big{)}^2\Big{]}dxdy.
$$
Using the partition of unity on $X$, the Dirichlet integral $D(\mathcal{F})$ of $\mathcal{F}$ over $\cup_jE_j\subset X$ is well-defined. If the integration is over a subsurface $Y$ of $X$, denote the corresponding Dirichlet integral over $(\cup_jE_j)\cap Y$ by $D_Y(\mathcal{F})$. 

\begin{definition}
A proper partial measured foliation $\mathcal{F}$ on $X$ is called an integrable foliation if $D(\mathcal{F})=D_X(\mathcal{F})<\infty$.
\end{definition}

Consider an integrable holomorphic quadratic differential $\varphi$ on a hyperbolic Riemann surface $X=\mathbb{H}/\Gamma$  where $\Gamma$ is of the first kind. If $w=u+iv$ is a natural parameter of $\varphi$ with the coordinate chart $U$, then $v:U\to\mathbb{R}$ defines the horizontal foliation $\mathcal{F}_{\varphi}$ of the differential $\varphi$. The horizontal foliation $\mathcal{F}_{\varphi}$ is a proper partial foliation because each non-singular horizontal trajectory of $\varphi$ when lifted to $\mathbb{H}$ accumulates to a single point in $\partial \mathbb{H}$ in each direction. The two limit points are different (see \cite{MardenStrebel}, \cite{Strebel}).

Note that for $\mathcal{F}_{\varphi}$ we have $E_i\equiv U_i$ for all $i$. 
Since $\int_U|\varphi (z)|dxdy=\int_{w(U)}dudv=\int_{w(U)}[(\frac{\partial v}{\partial u})^2+ (\frac{\partial v}{\partial v})^2]dudv$ and by the invariance of the Dirichlet integral under conformal maps, we get
$$
\int_U|\varphi (z)|dxdy=\int_{U}\Big{[}\Big{(}\frac{\partial v}{\partial x}\Big{)}^2+ \Big{(}\frac{\partial v}{\partial y}\Big{)}^2\Big{]}dxdy =D_U(\mathcal{F}_{\varphi}).
$$
Therefore if $\int_X|\varphi (z)|dxdy<\infty$ we have $D(\mathcal{F}_{\varphi})<\infty$, i.e. $\mathcal{F}_{\varphi}$ is an integrable partial foliation.
 
 We prove that the converse of the above statement is true in the following sense. 
 
 \begin{definition}
 A measured geodesic lamination $\mu$ on a Riemann surface $X$ is said to be realizable by an integrable partial foliation if there is an integrable partial foliation $\mathcal{F}$ such that $\mu_{\mathcal{F}}=\mu$.
 \end{definition}
 
Let $ML_{\mathrm{int}}(X)$ be the space of all measured geodesic laminations on the Riemann surface $X$ that are realizable by integrable partial foliations. Denote by $A (X)$  the space of all integrable holomorphic quadratic differentials on $X$.
 
 In the case of a compact surface of genus at least two, Hubbard-Masur \cite{HubbardMasur} and Kerckhoff \cite{Kerckhoff} proved that the map which assigns to each holomorphic quadratic differential its horizontal foliation is a bijection between the space of holomorphic quadratic differentials and homotopy classes of measured foliation of the surface, where measured foliations cover the surface. In other words, given a measured foliation on a compact surface there is a unique holomorphic quadratic differential whose horizontal foliation is homotopic to the given foliation. The space of homotopy classes of measured foliations is homeomorphic to the space of measured (geodesic) laminations on the surface by straightening the leaves of the foliation to hyperbolic geodesics (see \cite{Levitt}, \cite{Thurston}). Therefore, the result of Hubbard-Masur \cite{HubbardMasur} and Kerckhoff \cite{Kerckhoff} can be restated as the space of holomorphic quadratic differential on a compact surface of genus at least two is in a one-to-one correspondence with the space of measured laminations on the surface via straightening the leaves of the horizontal foliations.  
 
 We prove an analogous statement for infinite surfaces. 

\begin{thm}
\label{thm:main}
Let $X=\mathbb{H}/\Gamma$ be an infinite Riemann surface, where $\Gamma$ is a Fuchsian group of the first kind. The map
$
A(X)\to ML_{\mathrm{int}}(X)
$ defined by
$$\varphi\mapsto \mu_{\varphi}$$ obtained by straightening the trajectories of the horizontal foliation of $\varphi$ is a bijection.
\end{thm}

\begin{proof}
The map $A(X)\to ML(X)$ which assigns to $\varphi\in A(X)$ the geodesic lamination $\mu_{\varphi}$ obtained by straightening the horizontal foliation is well-defined and injective (see \cite[Theorem 5.2]{Saric-heights}). We established above that the image of $A(X)$ is in $ML_{\mathrm{int}}(X)$. It remains to prove that each $\mu\in ML_{\mathrm{int}}(X)$ is realized by the horizontal foliation of some $\varphi\in A(X)$. Sections \ref{sec:double} and \ref{sec:limit} are devoted to proving this statement.
\end{proof}

\subsection{The partial foliation on the double surface}
\label{sec:double}
Since $\Gamma$ is of the first kind, the Riemann surface $X=\mathbb{H}/\Gamma$ has a geodesic pants decomposition (see \cite{BS}). A geodesic pair of pants in the decomposition of $X$ can have at most two cusps on its boundary. In particular, it is not necessarily a compact subset of $X$. However, each pair of pants has a finite hyperbolic area. We form an exhaustion $\{ X_n\}_n$ of $X$ where each $X_n$ is a connected union of finitely many pairs of pants from the fixed pants decomposition. Then each $X_n$ is a finite area hyperbolic surface with a boundary consisting of some finite set of closed boundary geodesics and finitely many cusps of the pants decomposition. Closed boundary geodesics of pairs of pants in the pants decomposition will be called {\it cuffs} for short.

Let $\mathcal{F}$ be an integrable partial foliation of $X$ that represents $\mu\in ML_{\mathrm{int}}(X)$. 
If $\gamma$ is a simple closed geodesic in $X$, then the height $h_{\mathcal{F}}(\gamma )$ of $\gamma$ with respect to $\mathcal{F}$ is equal to the intersection $i(\mu ,[\gamma ])$, where $[\gamma ]$ is the homotopy class of $\gamma$ (see \cite{Saric-heights} and \cite{MardenStrebel1}). 


We modify the partial foliation $\mathcal{F}$ to obtain a partial foliation $\mathcal{F}_n$ on a Riemann subsurface $X_n'$ of $X$ that is homotopic to $X_n$. Let $\mathfrak{g}$ be a horizontal trajectory of $\mathcal{F}$. Consider a component $\mathfrak{g}_n$ of $\mathfrak{g}\cap X_n$. If $\mathfrak{g}_n$ has both of its endpoints on the same component of the boundary of $X_n$ and if it can be homotoped relative its endpoints to the corresponding boundary component, then the family of all components of horizontal leaves of $\mathcal{F}\cap X_n$ homotopic to $\mathfrak{g}_n$ relative the boundary component is nested. Consider only the subfamily of $\{ E_i\}_i$ that intersects $X_n$ and erase all the leaves of $\mathcal{F}\cap X_n$ that can be homotoped to a single boundary component modulo their endpoints. Denote the obtained family of sets $\{ E_j^n\}_j$. 
By the erased leaves' nesting property and the foliation's local product structure in $E_i$, each $E_j^n$ is obtained by deleting all the leaves below or above a single local leaf. This implies that the new family $\{ E_j^n\}_j$ is contained in closed Jordan domains $\{U_j^n\}_j$ which are subdomains of the family $\{ U_i\}_i$ satisfying the properties in the definition of a partial measured lamination where the functions are given by the restrictions of $v_i$. 

Each set $U_j^n$ intersects $X_n$, but it does not necessarily lie in $X_n$. We aim to define a new Riemann subsurface $X_n'$ and the charts $U_j'$ to lie in $X_n'$ and define a partial measured foliation of $X_n'$.  

Let $\gamma$ be a component of the boundary of $X_n$ and assume that $\gamma\cap U_j^n\neq\emptyset$ for at least one $U_j^n$. Each $U_j^n$ has a product structure from the Implicit Function Theorem. We add finitely many closed Jordan domains to the domains  $U_j^n$ that intersect $\gamma$ such that $\gamma$ is covered by the finitely many charts. The added domains are chosen small enough so that they do not intersect $X_{n-1}$, and we choose a product structure on them with the requirement that it is consistent on the overlaps. We choose the notion of vertical arcs based on a fixed product structure in each added chart.
Then there exists a simple closed step curve $\gamma'$ homotopic to $\gamma$ contained in the union of the charts covering $\gamma$, where a step curve consists of finitely many horizontal and vertical arcs. The proof of this fact is the same as for holomorphic quadratic differentials in Strebel \cite[Theorem 24.1]{Strebel}.

We define the subsurface $X_n'$ and a partial measured foliation $\mathcal{F}_n$ on $X_n'$ as follows.
\begin{definition}
\label{def:X_n'}
Let $\{\gamma_1,\ldots ,\gamma_{k_n}\}$ be the boundary geodesics of $X_n$ and $\{\gamma_1',\ldots ,\gamma_{k_n}'\}$ the step curves obtained using the above construction. 
Then $X_n'$ is the finite-area component of $X\setminus\{ \gamma_1',\ldots ,\gamma_{k_n}'\}$. \end{definition}


The bordered Riemann surfaces $X_n$ and $X_n'$ are homotopic in $X$.
The family of closed Jordan domains $\{ U_i\}$ is locally finite in $X$ by the Definition \ref{def:partial-fol}. Therefore, each $U_i$ is entirely contained in $X_n$ for all $n$ large enough depending on $i$ and $\{ X_n'\}_n$ is exhaustion of  $X$ by finite area Riemann subsurfaces. 

The set of connected components of the intersection of each $U_j^n$ with $\cup_{i=1}^{k_ n}\gamma_{i}'$ consists of step curves. There are only finitely many components of the intersection because the horizontal or vertical arcs of the set $\cup_{i=1}^{k_ n}\gamma_{i}'$ are on a definite distance from each other, and a single horizontal arc or a single vertical arc can intersect any of the Jordan domains covering $\gamma$ at most finitely many times. Since any infinite set in $U_j^n$ would converge in the closure $\bar{U}_j^n$ we cannot have infinitely components. 

The components of $U_j^n\cap (\cup_{i=1}^{k_ n}\gamma_{i}')$ divide $U_j^n$ into finitely many component charts and we additionally divide each component chart into finitely many charts by drawing horizontal lines of the division for each horizontal arc in a component of $U_j^n\cap (\cup_{i=1}^{k_ n}\gamma_{i}')$. We discard the components of $U_j^n\setminus (\cup_{i=1}^{k_ n}\gamma_{i}')$ which are not in $X_n'$.

Denote by $\{ U_h'\}_h$ the family of all the above components that lie in $X_n'$, including $U_j^n$ that are entirely contained in $X_n$. Denote by $E_h'\subset U_h'$ the sets obtained by the restrictions of the sets from the family $\{ E_j^n\}_j$.
  
 \begin{definition}
 \label{def:F_n} 
 The partial measured foliation $\mathcal{F}_n$ on $X_n'$ is given by the restriction of the functions $\{v_i\}_i$ to the sets $\{U_h'\}_h$ denoted by $\{ v_h'\}_h$, where the local leaves are given by the family $\{ E_h'\}_h$.
\end{definition}

The Riemann surface $X_n'$ is obtained by truncating $X$ along finitely many  simple closed step curves $\gamma_i'$ corresponding to the boundary geodesics of $X_n$. We use conformal maps of neighborhoods of the boundary points of $X_n'$ to the half-disks in the closure of $\mathbb{H}$ with the diameter on $\mathbb{R}$ that map the boundary points to the diameter. These conformal charts of the boundary make $X_n'$ into a bordered Riemann surface, and we form the double the Riemann surface $X_n'$ across its boundary, called $\widehat{X}_n'$.

Let  $\widehat{\mathcal{F}}_n$ be the partial measured foliation on $\widehat{X}_n'$ which equals 
$\mathcal{F}_n$ on $X_n'$ and the mirror image of $\mathcal{F}_n$ on $\widehat{X}_n'\setminus X_n'$. Each horizontal leaf of $\mathcal{F}_n$ 
that accumulates to the boundary of $X_n'$ in a given direction necessarily has a unique endpoint 
on the boundary of $X_n'$ by our construction. This property was the reason for introducing the new surface $X_n'$.

 If a horizontal leaf in $X_n'$ accumulates to the boundary of $X_n'$ 
 at both ends, it is then continued by its mirror image to make a closed leaf in $\widehat{X}_n'$.
Therefore $\widehat{\mathcal{F}}_n$  has closed leaves as well as non-closed leaves that are coming from the non-closed leaves of $\mathcal{F}$ contained in $X_n'$ as well as components of leaves of $\mathcal{F}$ in $X_n'$ with one ray accumulating to the boundary of $X_n'$ and the opposite ray accumulating in the interior $X_n'$. The leaves on $\widehat{\mathcal{F}}_n$ are differentiable curves away from the boundary cuffs of $X_n$ and possibly the boundaries of the Jordan domain charts for $\mathcal{F}$. 

\begin{lem}
\label{lem:proper}
The partial measured foliation $\widehat{\mathcal{F}}_n$ on the Riemann surface $\widehat{X}_n'$ is proper.
\end{lem}

\begin{proof}
We need to prove that each non-singular leaf $\mathfrak{f}$ of $\widehat{\mathcal{F}}_n$ lifts to a leaf with precisely two distinct endpoints on the ideal boundary of the universal cover of $\widehat{X}_n'$. The leaf $\mathfrak{f}$ is either completely contained in the interior of $X_n'$ or the interior of $\widehat{X}_n'\setminus X_n'$, or it is closed, or it is open and has accumulation points in both $X_n'$ and $\widehat{X}_n'\setminus X_n'$.

If $\mathfrak{f}$ is a leaf which is contained in the interior of $X_n'$ or the interior of $\widehat{X}_n'\setminus X_n'$, then its lift to the universal cover of $\widehat{X}_n'$ has precisely two distinct accumulation points on the ideal boundary because of the corresponding assumption on the leaves of $\mathcal{F}$. 

If $\mathfrak{f}$ is a closed leaf not in $X_n'$ nor in $\widehat{X}_n'\setminus X_n'$, then $\mathfrak{f}\cap X_n'$ is not homotopic to a boundary component of $X_n'$ relative to its endpoints by the construction of $\widehat{\mathcal{F}}_n$. Therefore $\mathfrak{f}\cap X_n'$ either connects two different boundary components of $X_n'$, or it joins the same boundary component while essentially intersecting some closed geodesic in the interior of $X_n'$, or it separates at least one puncture of $X_n'$ from the rest of $X_n'$. In the first case, the leaf $\mathfrak{f}$ essentially intersects the two boundary components of $X_n'$ that contain the endpoints of $\mathfrak{f}\cap X_n'$ and therefore $\mathfrak{f}$ lifts to a leaf with two distinct endpoints on the ideal boundary of the universal cover of $\widehat{X}_n'$. In the second case, since $\mathfrak{f}\cap X_n'$ essentially intersects a closed geodesic in the interior of $X_n'$, we conclude that $\mathfrak{f}$ also essentially intersects the same geodesic. Therefore the lift has two distinguished endpoints in this case as well. In the third case, the leaf $\mathfrak{f}$ separates the puncture mentioned above and its mirror image from the rest of the surface $\widehat{X}_n'$, and its lift has two distinguished ideal endpoints. 

The last case to consider is when $\mathfrak{f}$ has accumulation points in both $X_n'$ and $\widehat{X}_n'\setminus X_n'$. One ray of $\mathfrak{f}\cap X_n'$ accumulates to a point on the boundary of $X_n'$, and the opposite ray has accumulation points in the interior of $X_n'$. Let $\pi :\mathbb{H}\to\widehat{X}_n'$ be the universal covering and fix a component $\tilde{X}_n^c$ of $\pi^{-1}(X_n')$. Let $\tilde{\mathfrak{f}}$ be a single lift of $\mathfrak{f}\cap X_n'$ to $\tilde{X}_n^c$. Let $g$ be a geodesic on the boundary of $\tilde{X}_n^c$ that $\tilde{\mathfrak{f}}$ meets. The opposite ray of $\tilde{\mathfrak{f}}$ accumulates to a point in the ideal boundary of $\mathbb{H}$ by the corresponding assumption on $\mathcal{F}$. The reflection of $\tilde{X}_n^c$ in $g$ is a component of $\pi^{-1}(\widehat{X}_n'\setminus X_n')$. The reflection of the trajectory $\mathfrak{f}$ continues $\mathfrak{f}$ in the component of  $\pi^{-1}(\widehat{X}_n'\setminus X_n')$ to make a single lift of $\mathfrak{f}$ in $\mathbb{H}$. It is obvious that the reflection also has an endpoint on the ideal boundary of $\mathbb{H}$. Therefore each component of $\pi^{-1}(\mathfrak{f} )$ has two accumulation points as desired.
\end{proof}

Since the lift of each leaf in $\widehat{\mathcal{F}}_n$ has two endpoints on the ideal boundary of $\mathbb{H}$, the homotopy class of the partial foliation $\widehat{\mathcal{F}}_n$ on $\widehat{X}_n'$ induces a measured geodesic lamination $\mu_n$ on $\widehat{X}_n'$ which is invariant under the reflection in the boundary of $X_n'$. Then there is a measured foliation $\widehat{\mathcal{F}}_{\mu_n}$ on $\widehat{X}_n'$ whose straitghening  is the measured lamination $\mu_n$ (see \cite{Levitt}). By Hubbard-Masur \cite{HubbardMasur} or Kerckhoff \cite{Kerckhoff}, there is a unique holomorphic quadratic differential $\widehat{\varphi}_n$ whose horizontal foliation has the heights equal to the heights of $\widehat{\mathcal{F}}_{\mu_n}$. Let $\varphi_n$ be a holomorphic quadratic differential on $X_n'$ which is obtained by the restriction of $\widehat{\varphi}_n$ to $X_n'$.
 
\begin{lem}
\label{lem:lim-not-zero}
Let $X_n'$ be a finite area surface in the exhaustion of $X$ and $\varphi_n$ a holomorphic quadratic differential defined by the above process using the integrable partial foliation $\mathcal{F}$ on $X$. Then there exists a subsequence $\varphi_{n_k}$ of $\varphi_n$ that converges uniformly on compact sets to a non-zero integrable holomorphic quadratic differential $\varphi$ on $X$.
\end{lem} 

\begin{proof}
 By the definition, the heights of $\widehat{\varphi}_n$ and $\widehat{\mathcal{F}}_n$ are equal. We have that
$$
\int_{\widehat{X}_n'}|\widehat{\varphi}_n(z)|dxdy\leq D_{\widehat{X}_n'}(\widehat{\mathcal{F}}_n)
$$
by \cite[Theorem 7.5]{Bourque}. The key idea is to use the decomposition of the surface into strips, cylinders, and minimal components of the vertical foliation of $\varphi$ and then to use the length-area argument (see also \cite[Theorem 3.2]{MardenStrebel1}).
Since the argument uses integration over $\widehat{X}_n$, it is enough to assume that the partial foliation is differentiable off a countable union of differentiable arcs.
By the mirror symmetry of $\widehat{X}_n'$ and $\widehat{\mathcal{F}}_n$, we have that $D_{\widehat{X}_n'}(\widehat{\mathcal{F}}_n)=2D_{{X}_n'}(\mathcal{F}_n)$.
Since we obtained $\mathcal{F}_n$ by restricting the leaves of $\mathcal{F}$ to $X_n'$ and then erasing some of them, we have the inequality 
$ D_{{X}_n'}(\mathcal{F}_n)\leq D_X(\mathcal{F}).$ Finally, for all $n$, we obtain
$$
\int_{\widehat{X}_n'}|\widehat{\varphi}_n(z)|dxdy\leq 2D_{X}({\mathcal{F}})
$$

It follows that $\varphi_n:=\widehat{\varphi}_n|_{X_n'}$ has a subsequence that  converges uniformly on compact subsets of $X$ to an integrable holomorphic quadratic differential $\varphi$ with 
$$
\int_X|\varphi (z)|dxdy\leq 2D_{X}(\mathcal{F}).
$$

It remains to prove that $\varphi$ is not identically equal to zero. Let $\gamma$ be a simple closed geodesic of $X$ such that $h_{\mathcal{F}}(\gamma )=i(\mu ,\gamma )>0$. Given that the height is  the infimum over all closed curves homotopic to $\gamma$ in $X$, we immediately have
\begin{equation}
\label{eq:ineq-n}
h_{\mathcal{F}}(\gamma )\leq h_{\mathcal{F}, X_n'}(\gamma ).
\end{equation}
Since the partial foliation $\mathcal{F}_n$ on $X_n'$ is obtained by erasing the leaves of $\mathcal{F}$ that are homotopic to the boundary relative to their endpoints, $h_{\mathcal{F},X_n'}(\gamma )$ may be larger than $h_{\mathcal{F}_n}(\gamma )$. 

We first prove that the sequence  $h_{\mathcal{F}_n}(\gamma )$ has a  positive lower bound. Let $\gamma_1$ be a curve in $X_n'$ homotopic to $\gamma$. Then we have $\int_{\gamma_1\cap\mathcal{F}}|dv|\geq h_{\mathcal{F}}(\gamma )$, where $|dv|$ is the $1$-form corresponding to $\mathcal{F}$. Consider the set of all leaves of $\mathcal{F}\cap X_n'$ that can be homotoped to a boundary component relative to their endpoints. This set can be written as an at most countable union $\{ C_j\}_j$ of relatively open strips of leaves. Given $\epsilon >0$, we will show that $\gamma_1$ can be modified such that 
$$
\sum_j\int_{\gamma_1\cap\mathcal{F}\cap C_j}|dv|<\epsilon .
$$
Indeed, the complement in $X_n'$ of each $C_j$ has two components: the {\it inside} component, which is connected and can be homotoped to $\partial X_n'$, and the {\it outside} component, which is not simply connected (see Figure 1). Let $\tilde{C}_j$ be the union of $C_j$ and the inside component of $C_j$. The intersection of $\gamma_1$ with $\cup\tilde{C}_j$  has at most countably many components $\{\gamma_1^k\}_k$. If $\gamma_1^k\subset\tilde{C}_j$, we choose a leaf $\ell_k\subset C_j$ such that the integral over $\gamma_1^k\cap O_k$ is at most $\epsilon /2^{k}$, where $O_k$ is the outside region of  $C_j\setminus\ell_k$. Denote by $I_k$ the inside region $C_j\setminus\ell_k$. We replace each component of
$\gamma_1^k\cap I_k$ with an arc on $\ell_k$ having the same endpoints as the corresponding component of $\gamma_1^k\cap I_k$ (see Figure 1). We perform this process for each component $\gamma_1^k$ to obtain the new curve $\gamma_1'$ homotopic to $\gamma_1$. By the above choices, the total contribution to $\int_{\gamma_1'}|dv|$ of the leaves of $\mathcal{F}\setminus\mathcal{F}_n$ is less than $\sum_{k=1}^{\infty}\epsilon /2^k=\epsilon$. 

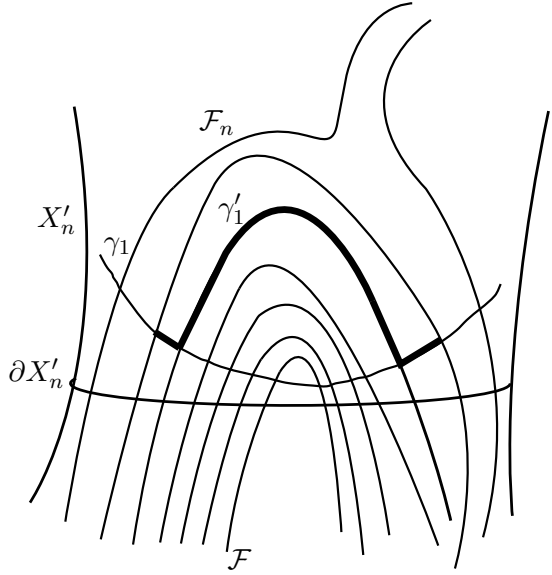
\begin{figure}[h]
\label{fig:approx-height}
\leavevmode \SetLabels
\L(.43*.0) $\mathcal{F}$\\
\L(.42*.62) $\gamma_1'$\\
\L(.4*.77) $\mathcal{F}_n$\\
\L(.3*.56) $\gamma_1$\\
\L(.2*.33) $\partial X_n'$\\
\L(.23*.6) $X_n'$\\
\endSetLabels
\begin{center}
\AffixLabels{\centerline{\epsfig{file =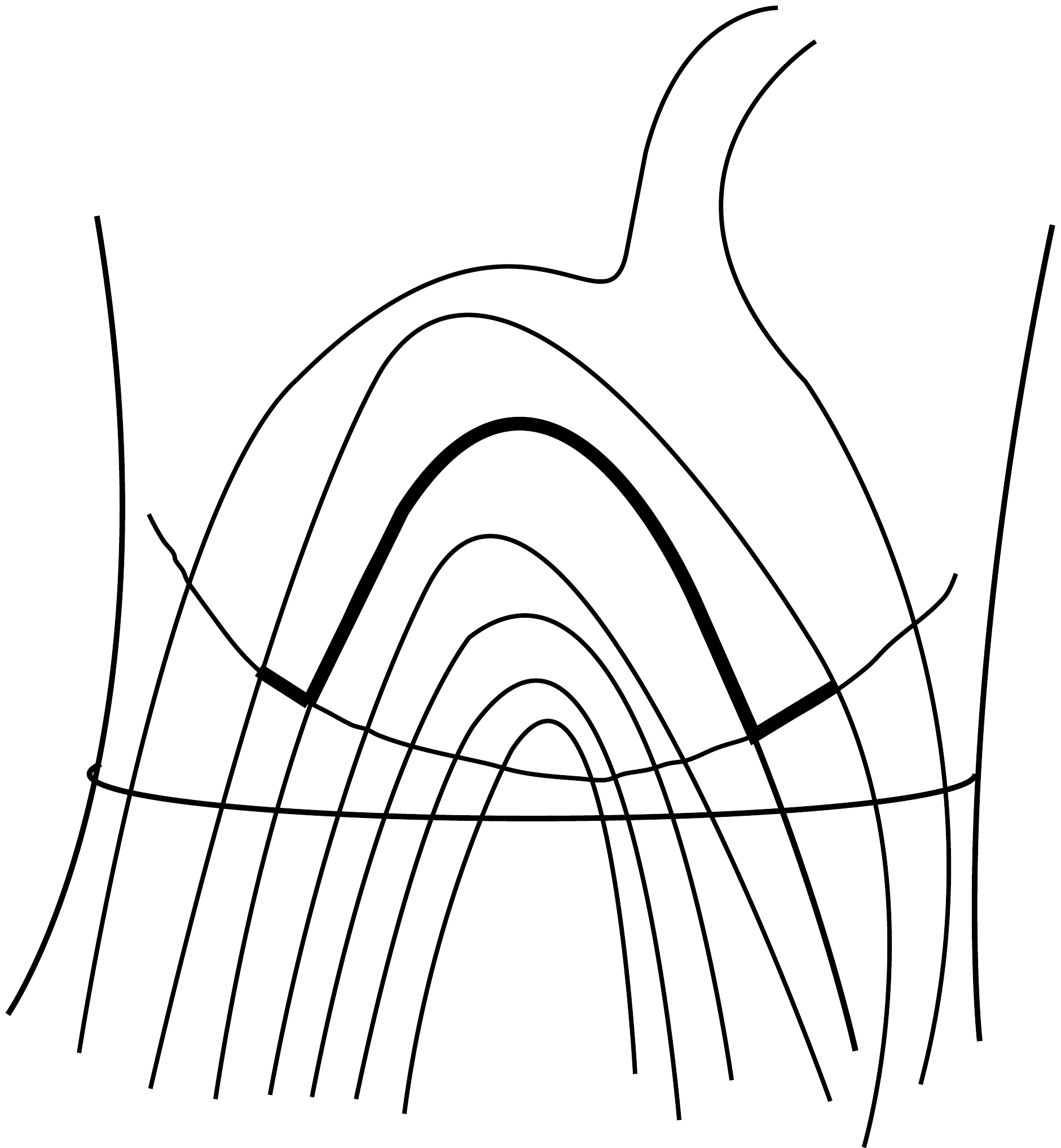,width=7.0cm,angle=0} }}
\vspace{-20pt}
\end{center}
\caption{Changing $\gamma_1$ to $\gamma_1'$. The bold line is the part that is changed.}
\end{figure}

If we denote by $|dv^n|$ the $1$-form for $\mathcal{F}_n$, the above proves that for each $\gamma_1\subset X_n'$ in the homotopy class of $\gamma$ there is $\gamma_1'$ also in the homotopy class of $\gamma$ such that
\begin{equation}
\label{eq:ineq-first}
\int_{\gamma_1'\cap\mathcal{F}}|dv|\leq\int_{\gamma_1'\cap\mathcal{F}}|dv^n|+\epsilon \leq\int_{\gamma_1\cap\mathcal{F}}|dv^n| +\epsilon .
\end{equation}
Since (\ref{eq:ineq-first}) holds for all $\gamma_1$ in $X_n'$,
 (\ref{eq:ineq-n}) gives
\begin{equation}
\label{eq:height-ineq}
h_{\mathcal{F}}(\gamma )-\epsilon\leq h_{\mathcal{F}_n}(\gamma ) .
\end{equation}
The partial foliation $\widehat{\mathcal{F}}_n$  is invariant under the reflection in the boundary of $X_n'$
which implies that the holomorphic quadratic differential $\widehat{\varphi}_n$  realizing  $\widehat{\mathcal{F}}_n$  is also invariant under the reflection in the boundary of $X_n'$. It follows that 
\begin{equation}
\label{eq;height-symm}
h_{\mathcal{F}_n}(\gamma )=h_{\widehat{\mathcal{F}}_n}(\gamma )=h_{\widehat{\varphi}_n}(\gamma )=h_{\varphi_n}(\gamma ).
\end{equation}

Fix $\gamma_0\subset X$ homotopic to $\gamma$. Since $\varphi_n\to\varphi$ as $n\to\infty$ uniformly on compact subsets of $X$, there exists $n_0$ such that 
\begin{equation}
\label{eq:lim-app}
h_{\varphi_n}(\gamma )\leq \int_{\gamma_0}|Im(\sqrt{\varphi (z)}dz)|+\epsilon
\end{equation}
for all $n\geq n_0$.

By (\ref{eq:height-ineq}), (\ref{eq;height-symm}) and (\ref{eq:lim-app}) we obtain
\begin{equation}
\label{eq:hol-partial}
\int_{\gamma_0}|Im(\sqrt{\varphi (z)}dz)|\geq h_{\mathcal{F}}(\gamma )-2\epsilon .
\end{equation}
When we choose $\epsilon <\frac{1}{2}h_{\mathcal{F}}(\gamma )$
we conclude that $\varphi$ is not identically equal to zero.
\end{proof}

\subsection{The heights of the limiting quadratic differential}
\label{sec:limit}
In the previous subsection, we constructed a sequence of finite area surfaces $X_n'$ which approximate $X$ and constructed a sequence of holomorphic quadratic differentials $\varphi_n$ on $X_n'$ whose horizontal foliations are homotopic to a sub-foliation of the restriction of the foliation $\mathcal{F}$ to $X_n'$. In Lemma \ref{lem:lim-not-zero}, we established that there is a subsequence $\varphi_{n_k}$ which converges uniformly on compact subsets to a non-trivial integrable holomorphic quadratic differential $\varphi$ on $X$. It remains to prove that the horizontal foliation of $\varphi$ is equivalent to $\mathcal{F}$.

By \cite[Lemma 3.3 and Theorem 4.1]{Saric20}, to prove that the horizontal foliation of $\varphi$ induces $\mu$ it is enough to prove that $h_{\varphi}(\gamma )=i(\mu ,\gamma )$ for all simple closed geodesics $\gamma$ on $X$. This is equivalent to $h_{\varphi}(\gamma )=h_{\mathcal{F}}(\gamma )$  for all simple closed geodesics $\gamma$ by the definition of $\mathcal{F}$. 

\begin{lem}
\label{lem:heights=}
Let $X=\mathbb{H}/\Gamma$ with $\Gamma$ of the first kind and let $\mathcal{F}$ be an integrable partial foliation on $X$. The integrable holomorphic quadratic differential $\varphi$ on $X$ constructed above satisfies
$$
h_{\varphi}(\gamma )=h_{\mathcal{F}}(\gamma )
$$
for all simple closed curves $\gamma$ on $X$.
\end{lem}

\begin{proof}
 Since (\ref{eq:hol-partial}) holds for any $\epsilon >0$, by letting $\epsilon\to 0$ we get
 $$
 \int_{\gamma_0}|Im(\sqrt{\varphi (z)}dz)|\geq h_{\mathcal{F}}(\gamma ).
 $$
 By taking the infimum over all $\gamma_0$ homotopic to $\gamma$, we obtain
 $$
  h_{\varphi}(\gamma )\geq h_{\mathcal{F}}(\gamma ).
 $$

To obtain the opposite inequality, note that by \cite[Theorem 24.7]{Strebel} 
$$
\lim_{k\to\infty}h_{\varphi_{n_k}}(\gamma )=h_{\varphi}(\gamma )
$$
and by the definition
$$
\lim_{k\to\infty}h_{\mathcal{F},X_{n_k}}(\gamma )=h_{\mathcal{F}}(\gamma ).
$$

Since
$$
h_{\varphi_{n_k}}(\gamma )=h_{\mathcal{F}_{n_k}}(\gamma )\leq h_{\mathcal{F},X_{n_k}}(\gamma )
$$
for any closed curve $\gamma$ on $X$, by letting $k\to\infty$ in the above inequality, we get
$$
h_{\varphi}(\gamma )\leq h_{\mathcal{F}}(\gamma ).
$$ 
The lemma is established.
\end{proof}

\noindent {\it End of the proof of Theorem \ref{thm:main}.} Given an integrable partial foliation $\mathcal{F}$, we form a sequence of quadratic differentials $\varphi_{n}$ on the exhaustion $X_n'$ of $X$ as in \S \ref{sec:double}. Then Lemma \ref{lem:lim-not-zero} implies that there exists a subsequence $\varphi_{n_k}$ which converges uniformly on compact subsets to a non-trivial integrable holomorphic quadratic differential $\varphi$ on $X$. By Lemma \ref{lem:heights=}, the heights of $\varphi$ are equal to the heights of $\mathcal{F}$. Therefore the horizontal foliation of $\varphi$ realizes the measured lamination $\mu$. The quadratic differential $\varphi$ is unique by \cite[Theorem 5.3]{Saric-heights}. 

\subsection{The convergence in $A(X)$}
The space of integrable holomorphic quadratic differentials is equipped with the topology induced by the $L^1$-norms. It is helpful to describe this topology in terms of sequential convergence.

\begin{lem} 
\label{lem:L^1-conv}
A sequence $\varphi_n\in A(X)$ converges to $\varphi\in A(X)$ in the $L^1$-norm if and only if 
\begin{enumerate}
\item
$\varphi_n(z)$ converges to $\varphi (z)$ uniformly on compact subsets of $X$, and 
\item $\limsup_{n\to\infty}\int_X|\varphi_n|\leq \int_X|\varphi |$.
\end{enumerate}
\end{lem}
\begin{proof}
To prove the only if direction, note that the inequality $0\leq |\varphi -\varphi_n|+|\varphi |-|\varphi_n|\leq 2|\varphi |$ allows us to apply Lebesgue's Dominated Convergence Theorem to $|\varphi -\varphi_n|+|\varphi |-|\varphi_n|$ which gives $0\leq \lim_{n\to\infty}\int_X|\varphi -\varphi_n|=\lim_{n\to\infty}\int_X|\varphi_n|-\int_X|\varphi |\leq 0$. For the if direction, the uniform convergence on compact sets follows from the Cauchy Theorem for holomorphic function and the inequality $||\varphi |-|\varphi_n||\leq |\varphi -\varphi_n|$ implies $\int_X|\varphi_n|\to\int_X|\varphi |$ as $n\to\infty$. 
\end{proof}

By definition, we will consider the topology on $ML_{\mathrm{int}}(X)$ to be induced by the straightening map and the $L^1$-topology on $A(X)$. It seems to be difficult to recover the topology on $ML_{\mathrm{int}}(X)$ in terms of the heights and Dirichlet integrals alone since the Dirichlet integrals depend on the choice of the realizations of $\mu\in ML_{\mathrm{int}}(X)$. 

We can recover the topology of uniform convergence on compact sets in terms of the Dirichlet integrals of the realizations.

\begin{definition}
We say that $\lim_{n\to\infty}\mu_n=\mu$ for $\mu_n,\mu\in ML_{\mathrm{int}}(X)$ if there exist integrable partial foliations $\mathcal{F}_n, \mathcal{F}$ that realize $\mu_n,\mu$ such that
\begin{enumerate}
\item $\lim_{n\to\infty}h_{\mathcal{F}_n}(\gamma )=h_{\mathcal{F}}(\gamma )$ for each simple closed geodesic $\gamma$ in $X$, and
\item $D(\mathcal{F}_n)\leq M<\infty$ for all $n$.
\end{enumerate}
\end{definition}

\begin{prop}
\label{prop:unif_conv}
The inverse of the straightening map
$$
ML_{\mathrm{int}}(X)\to A(X)
$$ 
is continuous for the topology on $ML_{\mathrm{int}}(X)$ induced by the above definition and the topology of uniform convergence on compact sets on $A(X)$.
\end{prop}

\begin{proof}
Let $\varphi_n,\varphi \in A(X)$ realize $\mu_n,\mu\in ML_{\mathrm{int}}(X)$, respectively. Since $\int_X|\varphi_n|\leq D(\mathcal{F}_n)\leq M$, we conclude that a subsequence of $\varphi_n$ converges uniformly on compact subsets to a holomorphic quadratic differential $\psi\in A(X)$. Since the heights of $\mathcal{F}_n$  are equal to the heights of $\varphi_n$ and the heights of $\mathcal{F}$  are equal to the heights of $\varphi$, it follows that $\psi =\varphi$ by the Heights Theorem \cite{Saric-heights}. Therefore every subsequence of $\{\varphi_n\}_n$ has a subsequence converging to the same limit $\varphi$ uniformly on compact subsets. Thus the whole sequence converges to $\varphi$ as well.
\end{proof}

\section{The vertical foliation of quadratic differentials and harmonic measure of the boundary at infinity}

A Riemann surface $X=\mathbb{H}/\Gamma$ is said to be {\it parabolic} if it does not support a Green's function-i.e., there is no positive harmonic function on $X$ that has a single singularity of the form $-\log |z-\zeta |$ in a neighborhood of a single point $\zeta\in X$ expressed in a local chart whose values are zero at the ideal boundary(see Ahlfors-Sario \cite{AhlforsSario}). This definition gives a natural classification of infinite Riemann surfaces: $X$ is of parabolic type, i.e., $X\in O_G$, or $X$ is not of parabolic type, i.e., $X\notin O_G$.

The surfaces of parabolic type enjoy many favorable function theoretic, geometric, and dynamical properties. They were studied by many authors (for example, see \cite{BHS} for references).
The following is a list of some of the properties that are equivalent to $X=\mathbb{H}/\Gamma$ being parabolic, which was established  by Ahlfors-Sario \cite{AhlforsSario}, Hopf-Tsuji-Sullivan \cite{Tsuji}, \cite{Sullivan} , Astala-Zinsmeister \cite{Astala-Zinsmeister}, Bishop \cite{Bishop}:

\begin{itemize}
\item the geodesic flow for the hyperbolic metric on the unit tangent bundle of $X$ is ergodic
\item the Poincar\'e series for $\Gamma$ diverges
\item the harmonic measure of the boundary at infinity relative to a compact subsurface is zero
\item the group $\Gamma$ has Bowen property
\item the Brownian motion is recurrent
\item the surface $X$ does not support a non-constant positive subharmonic function
\end{itemize}

If an infinite Riemann surface $X=\mathbb{H}/\Gamma$ is not equal to its convex core, or if $\Gamma$ is of the second kind, then $X$ is not parabolic. This is a simple consequence of the fact that $X$ contains a hyperbolic funnel or a hyperbolic half-plane (see \cite{BS}). Indeed, in the case of a funnel or a half-plane, there is an open subset of the unit tangent bundle on which the geodesic flow escapes to infinity.

Recall that a horizontal trajectory of a holomorphic quadratic differential $\varphi$ on $X$ is called a {\it cross-cut} if it leaves every compact subset of $X$ at both ends (see \cite{Strebel}). It is a simple matter of constructing integrable holomorphic quadratic differentials on finite surfaces with funnels such that the set of cross-cuts has a non-zero measure.

In this paper, we assume that $\Gamma$ is of the first kind. The surface $X=\mathbb{H}/\Gamma$ may not be parabolic even in this case. In the case when $X$ is parabolic, Marden and Strebel (see \cite{MardenStrebel1} and \cite[Theorem 24.4]{Strebel}) proved the following theorem.

\begin{thm}[Marden-Strebel]
\label{thm:par-int}
Let $X=\mathbb{H}/\Gamma$ be an infinite Riemann surface of parabolic type and $\varphi$ a non-zero integrable holomorphic quadratic differential on $X$. Then the set of horizontal trajectories of $\varphi$ that are cross-cuts has zero area.
\end{thm}

The above theorem establishes the behavior of the horizontal trajectories of an integrable holomorphic quadratic differential on a parabolic Riemann surface. To complete this understanding, we consider what happens on any non-parabolic Riemann surface $X=\mathbb{H}/\Gamma$ with $\Gamma$ of the first kind. We prove

\begin{thm}
\label{thm:int-non-par-cross-cut}
Let $X=\mathbb{H}/\Gamma$ be an infinite Riemann surface not of parabolic type with $\Gamma$ of the first kind. Then there exists a non-zero integrable holomorphic quadratic differential $\varphi$ on $X$ such that the set of horizontal trajectories of $\varphi$ that are cross-cuts covers $X$ up to a set of zero area.
\end{thm}

\begin{proof}
Let $\{ X_n\}_{n=0}^{\infty}$ be an exhaustion of $X$ by finite area surfaces with analytic boundaries such that $X_0$ is a geodesic pair of pants. Let $u_n$ solve the Dirichlet problem on $X_n\setminus X_0$ with boundary values $0$ on $\partial X_0$ and $1$ on $\partial X_n$.
Namely, $u_n$ is 
 a harmonic function $u_n:X_n\setminus X_0\to\mathbb{R}$ such that $u_n|_{\partial X_0}=0$ and $u_n|_{\partial X_n}=1$.  By Ahlfors-Sario \cite[pages 204-205]{AhlforsSario}, since $X$ is not of parabolic type, the sequence $u_n$ has a subsequence $u_{n_k}$ that converges to a non-constant harmonic function $u$ defined on $X\setminus X_0$ such that $0<u<1$  with boundary values $0$ on $\partial X_0$.

In a neighborhood of any point $z\in X\setminus X_0$, the harmonic conjugate $u^{*}$ of $u$ is well-defined. Let $U(z)=u(z)+iu^*(z)$  be the corresponding holomorphic function which is defined only locally. By the Cauchy-Riemann equations, we have that $dU(z)=(\frac{\partial u}{\partial x}-i\frac{\partial u}{\partial y})dz$ is a well-defined holomorphic $1$-form on the interior of $X\setminus X_0$. Let $\varphi (z)dz^2=(dU(z))^2$ be the corresponding  holomorphic quadratic differential on $X\setminus X_0$. The set of horizontal arcs of $\varphi$
is
the set of all $(u^*)^{-1}(c)$ for $c$ real constant. A maximal curve obtained by concatenating horizontal arcs is a {\it horizontal trajectory} of $\varphi$. It is called {\it regular} if it does not accumulate at a zero of $\varphi$, i.e., point with $U'(z)=0$.

Analogously, we define $u_n^*$ to be the local harmonic conjugate of $u_n$ and $U_n=u_n+iu_n^*$ the corresponding local holomorphic map in $X_n\setminus X_0$. Let $\varphi_ndz^2=(dU_n)^2$ be the corresponding holomorphic quadratic differential.

By the definition, we have that $|\varphi (z)|=|\nabla u(z)|^2$ on $X\setminus X_0$ and $|\varphi_n (z)|=|\nabla u_n(z)|^2$ on $X_n\setminus X_0$. By Fatou's lemma, we have that
\begin{equation}
\label{eq:fatou}
\iint_{X\setminus X_0}|\nabla u(z)|^2dxdy\leq\liminf_{n\to\infty}\iint_{X_n\setminus X_0}|\nabla u_n(z)|^2dxdy.
\end{equation}
A standard consideration of the natural parameter gives us that the set of horizontal trajectories of $\varphi_n$ is dividing $X_n\setminus X_0$ into finitely many rectangles in the natural parameter (see Strebel \cite[Chapter IV, \S 13]{Strebel}). Since $u_n$ is increasing from $0$ to $1$ along the horizontal trajectories, the $|\sqrt{\varphi_n(z)}dz|$-length of each horizontal trajectory is $1$. Then integrating along horizontal trajectories and using Fubini's theorem, we get
\begin{equation}
\label{eq:fubini}
\iint_{X_n\setminus X_0}|\nabla u_n(z)|^2dxdy=\int_{\partial X_0}|du_n^*(z)|.
\end{equation}
Since $u_n\to u$ as $n\to\infty$ uniformly on compact subsets of the closure of $X_n\setminus X_0$ and since $\partial X_0$ consists of finitely many analytic arcs, we can normalize the harmonic conjugates $u_n^*$ of $u_n$ such that they converge to the harmonic conjugate $u^*$ of $u$ uniformly on compact subsets covering $\partial X_0$. 

Then we have
$$
\lim_{n\to\infty} \int_{\partial X_0}|du_n^*(z)|=\int_{\partial X_0}|du^*(z)|
$$
and by (\ref{eq:fatou}) and (\ref{eq:fubini}) we conclude that $\varphi$ is integrable on $X\setminus X_0$.

Except for an area zero set, the surface $X\setminus X_0$ is foliated by horizontal trajectories of $\varphi$ that are either closed, recurrent, or cross-cuts (see Strebel \cite[\S 13]{Strebel}). By the existence of the natural coordinates around any regular point, the harmonic function $u$ is strictly monotonic on each regular horizontal trajectory. Therefore a regular trajectory cannot be closed since $u$ would repeat itself. Similarly, a regular trajectory cannot be recurrent since it would intersect a vertical trajectory more than once, and the value of $u$ would repeat on a single trajectory. The remaining possibility is that all horizontal trajectories are cross-cuts except a set of zero area. If a cross-cut goes from $\partial X_0$ to itself, then $u$ would have a maximum on the cross-cut in the interior of $X\setminus X_0$ while the values of $u$ would converge to $0$ towards both ends of the cross-cut. Thus $u$ again cannot be monotone on the cross-cut.

Therefore a horizontal leaf in $X\setminus X_0$ (up to a set of measure zero) is a cross-cut that connects $\partial X_0$ to $\partial X$, or it connects $\partial X$ to itself. The set of horizontal leaves that start at $\partial X_0$ has non-zero (linear) $|du^*|$-measure, and it must accumulate to $\partial X$ by the above. We define a partial foliation $\mathcal{F}_u$ on $X\setminus X_0$ to consist of maps $u^*$ in the natural parameter of $\varphi$ restricted to the union of cross-cut horizontal leaves that connect $\partial X_0$ to $\partial X$. Then
$$
0<D_{X\setminus X_0}(\mathcal{F}_u)\leq\iint_{X\setminus X_0}|\varphi (z)|dxdy.
$$

We extend $\mathcal{F}_u$ to a partial foliation on the whole surface $X$ that only has cross-cut horizontal trajectories up to a set of zero area. Recall that $X_0$ has three boundary components (cuffs) with at least one cuff on the boundary of the component of $X\setminus X_0$ that contains $\mathcal{F}_u$. 

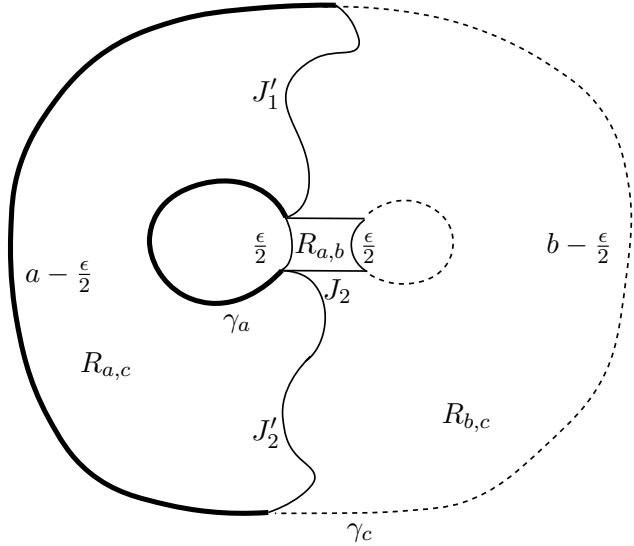
\begin{figure}[h]
\label{fig:approx-height}
\leavevmode \SetLabels
\L(.5*.05) $\gamma_c$\\
\L(.22*.34) $R_{a,c}$\\
\L(.162*.5) $a-\frac{\epsilon}{2}$\\
\L(.37*.42) $\gamma_a$\\
\L(.4*.55) $\frac{\epsilon}{2}$\\
\L(.514*.55) $\frac{\epsilon}{2}$\\
\L(.445*.55) $R_{a,b}$\\
\L(.6*.25) $R_{b,c}$\\
\L(.71*.55) $b-\frac{\epsilon}{2}$\\
\L(.4*.22) $J_2'$\\
\L(.4*.82) $J_1'$\\
\L(.475*.47) $J_2$\\
\endSetLabels
\begin{center}
\AffixLabels{\centerline{\epsfig{file =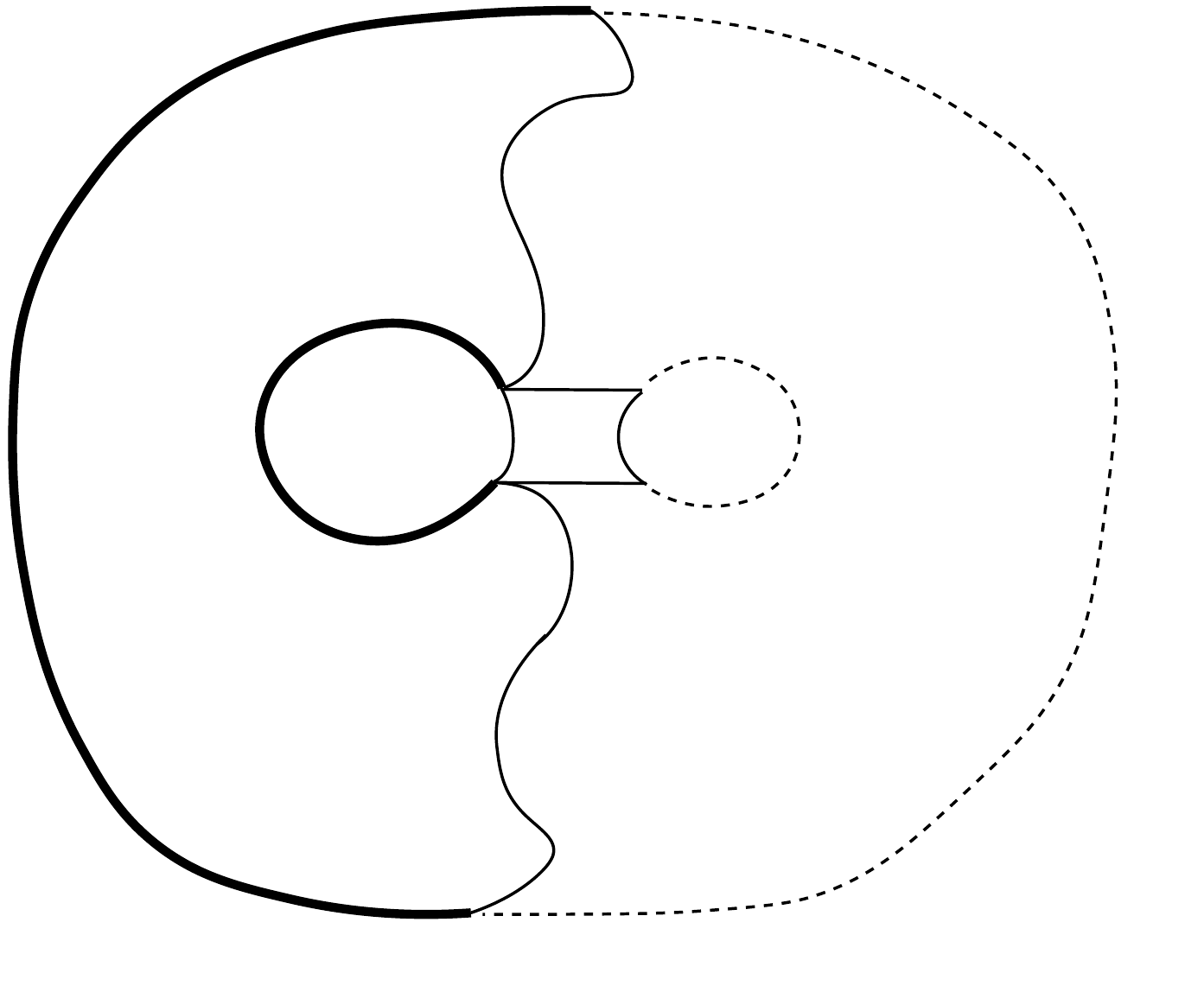,width=9.0cm,angle=0} }}
\vspace{-20pt}
\end{center}
\caption{$c\leq a+b$, $\epsilon =a+b-c$.}
\end{figure}

Assume first that all three cuffs of $X_0$ meet the leaves of the partial foliation $\mathcal{F}_u$. Let $a\leq b\leq c$ be the transverse measures of the boundary components $\gamma_a$, $\gamma_b$ and $\gamma_c$ of $X_0$, respectively. We divide it into two cases: $c\leq a+b$ or $c>a+b$. In the case $c\leq a+b$  we let $\epsilon=a+b-c$. Choose  arcs $I_a\subset\gamma_a$ and $I_b\subset \gamma_b$ both of transverse measure $\epsilon /2$. We connect endpoints of $I_a$ and $I_b$ with two analytic arcs $J_1$ and $J_2$  orthogonal to $\gamma_a$ and $\gamma_b$ at their endpoints such that $J_1\cup J_2\cup I_a\cup I_b$ is a boundary of a simply connected domain $R_{a,b}$ in $X_0$ (see Figure 2). We choose  subarcs $I_a'$ and $I_c'$  of $\gamma_a\setminus I_a$ and $\gamma_c$ that both have transverse measure $a-\frac{\epsilon}{2}$. We connect the endpoints of $I_a'$ and $I_c'$ by orthogonal analytic arcs $J_1',J_2'$ to form a simply connected region $R_{a,c}$ disjoint from $R_{a,b}$ with boundary $I_c'\cup I_a'\cup J_1'\cup J_2'$. Finally, we denote by $R_{b,c}$ the simply connected region $X_0\setminus (R_{a,b}\cup R_{b,c})$.  We map $R_{a,b}$, $R_{a,c}$ and $R_{b,c}$ by conformal mappings to Euclidean rectangles such that $J_i$ and $J_i'$ are mapped onto horizontal boundary sides, and we keep the same notation for the rectangles. 

The transverse measure on the vertical sides is pushed-forward to the transverse measure of the rectangles $R_{a,b}$, $R_{a,c}$ and $R_{b,c}$, and it is given by $\int_{*}h(y)dy$, where $h$ is a differentiable function of the Euclidean height $y$ and the total integrals over corresponding vertical sides are equal. We define $u^{*}(y)=\int_0^yh(s)ds$ where $0$ is the lowest point of the vertical side. The function $u^*(y)$, defined only on the vertical sides of the rectangles, gives the transverse measure of the arc from $0$ to $y$. We denote by $U^*$ the solution to the mixed Dirichlet-Neumann boundary value problem in each rectangle such that $U^*$ limits to $u^*$ on the vertical sides and the normal derivative of $U^*$ is zero on the horizontal sides. Note that the Dirichlet integral of $U^*$ over the rectangles is finite.

We complete the partial foliation $\mathcal{F}_u$ by attaching pre-images of the arcs $(U^{*})^{-1}(const)$ in the rectangles $R_{a,b}$, $R_{a,c}$ and $R_{b,c}$. The transverse measure to the foliation of $X_0$ agrees with the transverse measure of $\mathcal{F}_u$ on $\partial X_0$, and the Dirichlet integral on $X_0$ is finite. Therefore we obtain a partial measured foliation $\widehat{\mathcal{F}}_u$ of $X$ with $D_X(\widehat{\mathcal{F}}_u)<\infty$. 

It remains to prove that $\widehat{\mathcal{F}}_u$ is a proper foliation. Let $\ell$ be a leaf of $\widehat{\mathcal{F}}_u$. It is divided by $\ell\cap X_0$ into two rays on $X\setminus X_0$. Since the rays are not closed trajectories of $\varphi$, it follows that lifts to the universal cover of the two rays have different endpoints. Thus $\widehat{\mathcal{F}}_u$ is an integrable proper partial measured foliation of $X$ and Theorem \ref{thm:main} implies that there exists $\varphi_u\in A(X)$ that realizes $\widehat{\mathcal{F}}_u$. The trajectories of $\widehat{\mathcal{F}}_u$ are cross-cuts, and therefore all the trajectories of $\varphi_u$ are also cross-cuts which proves the theorem in the case $c\leq a+b$.

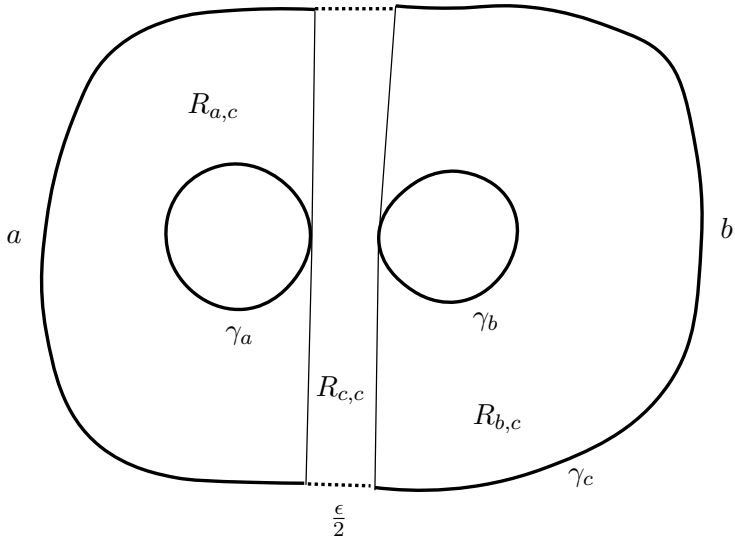
\begin{figure}[h]
\label{fig:approx-height}
\leavevmode \SetLabels
\L(.7*.15) $\gamma_c$\\
\L(.45*.07) $\frac{\epsilon}{2}$\\
\L(.6*.425) $\gamma_b$\\
\L(.86*.58) $b$\\
\L(.6*.25) $R_{b,c}$\\
\L(.34*.4) $\gamma_a$\\
\L(.3*.8) $R_{a,c}$\\
\L(.435*.3) $R_{c,c}$\\
\L(.11*.57) $a$\\
\endSetLabels
\begin{center}
\AffixLabels{\centerline{\epsfig{file =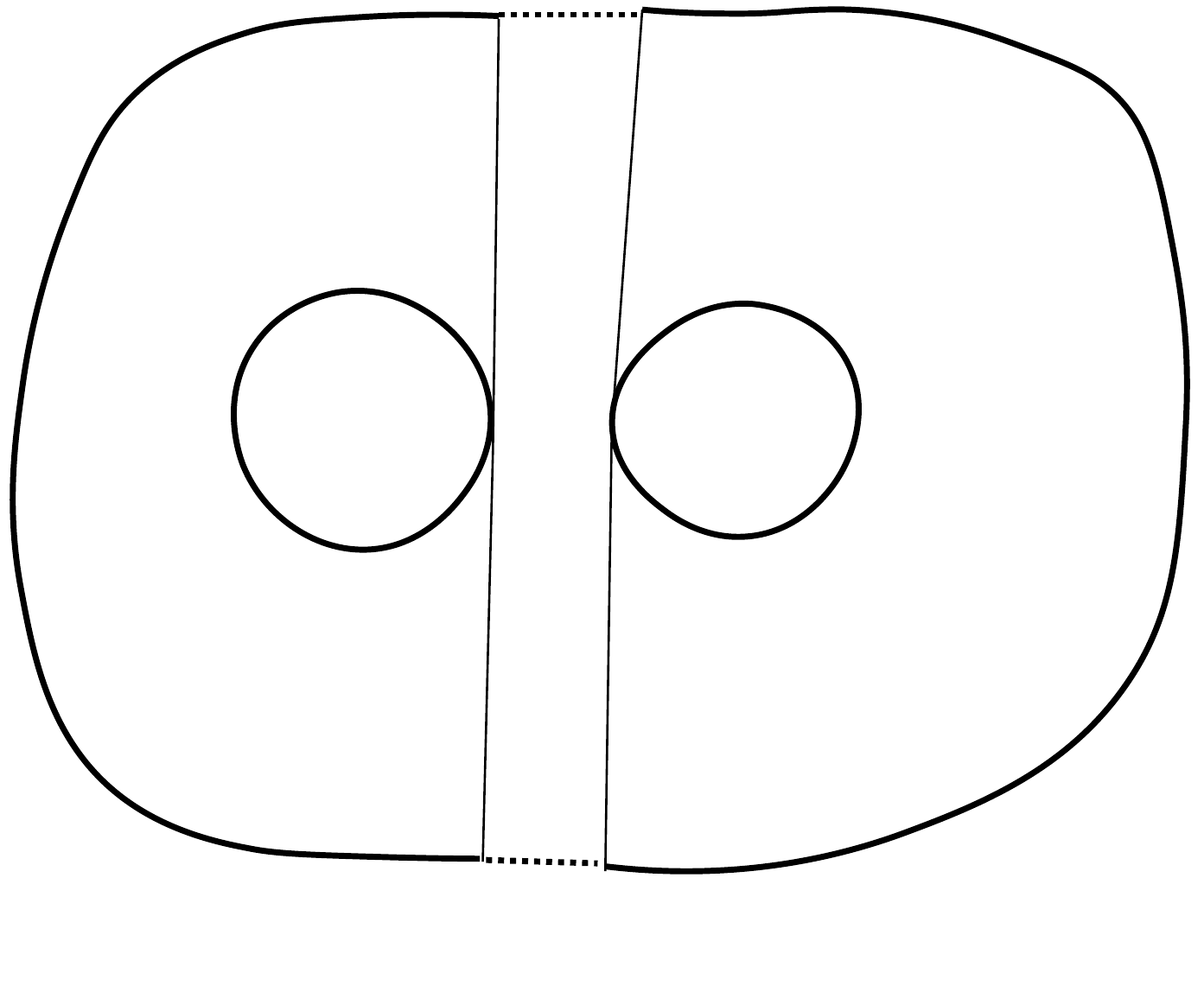,width=9.0cm,angle=0} }}
\vspace{-20pt}
\end{center}
\caption{$c> a+b$, $\epsilon =c-a-b$.}
\end{figure}

If $c>a+b$, then we refer to Figure 3 for the construction of the extension of $\widehat{\mathcal{F}}_u$ on $X_0$. All the steps in the proof are analogous to the previous case.

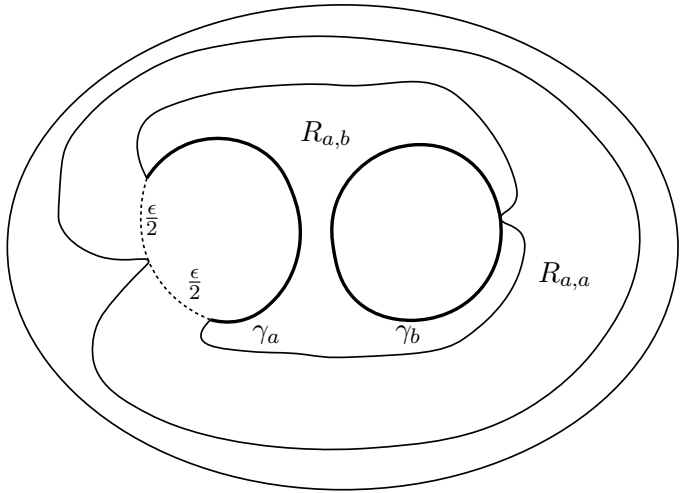
\begin{figure}[h]
\label{fig:approx-height}
\leavevmode \SetLabels
\L(.4*.4) $\gamma_a$\\
\L(.55*.4) $\gamma_b$\\
\L(.45*.75) $R_{a,b}$\\
\L(.7*.5) $R_{a,a}$\\
\L(.285*.6) $\frac{\epsilon}{2}$\\
\L(.33*.487) $\frac{\epsilon}{2}$\\
\endSetLabels
\begin{center}
\AffixLabels{\centerline{\epsfig{file =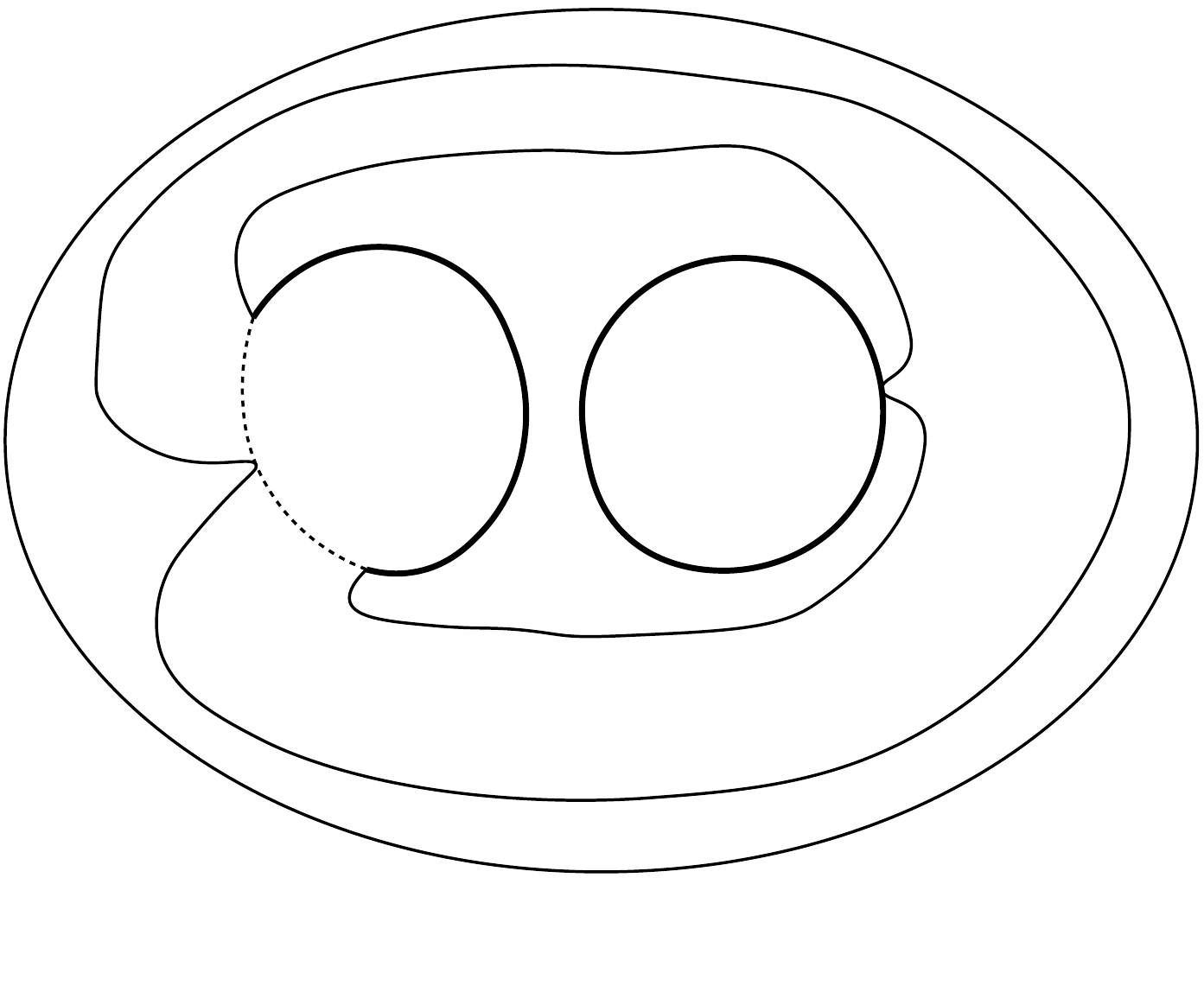,width=9.0cm,angle=0} }}
\vspace{-20pt}
\end{center}
\caption{$c=0$, $a\geq b$, $\epsilon =a-b$.}
\end{figure}

If only two boundary components of $X_0$ meet $\mathcal{F}_u$, then we refer to Figure 4 for the construction of the partial foliation $\widehat{\mathcal{F}}_u$. If only one boundary component of $X_0$ meets $\mathcal{F}_u$, then we refer to Figure 5 for the construction of the partial foliation $\widehat{\mathcal{F}}_u$.

\begin{figure}[h]
\label{fig:approx-height}
\leavevmode \SetLabels
\L(.38*.4) $\gamma_a$\\
\L(.4*.8) $R_{a,a}$\\
\endSetLabels
\begin{center}
\AffixLabels{\centerline{\epsfig{file =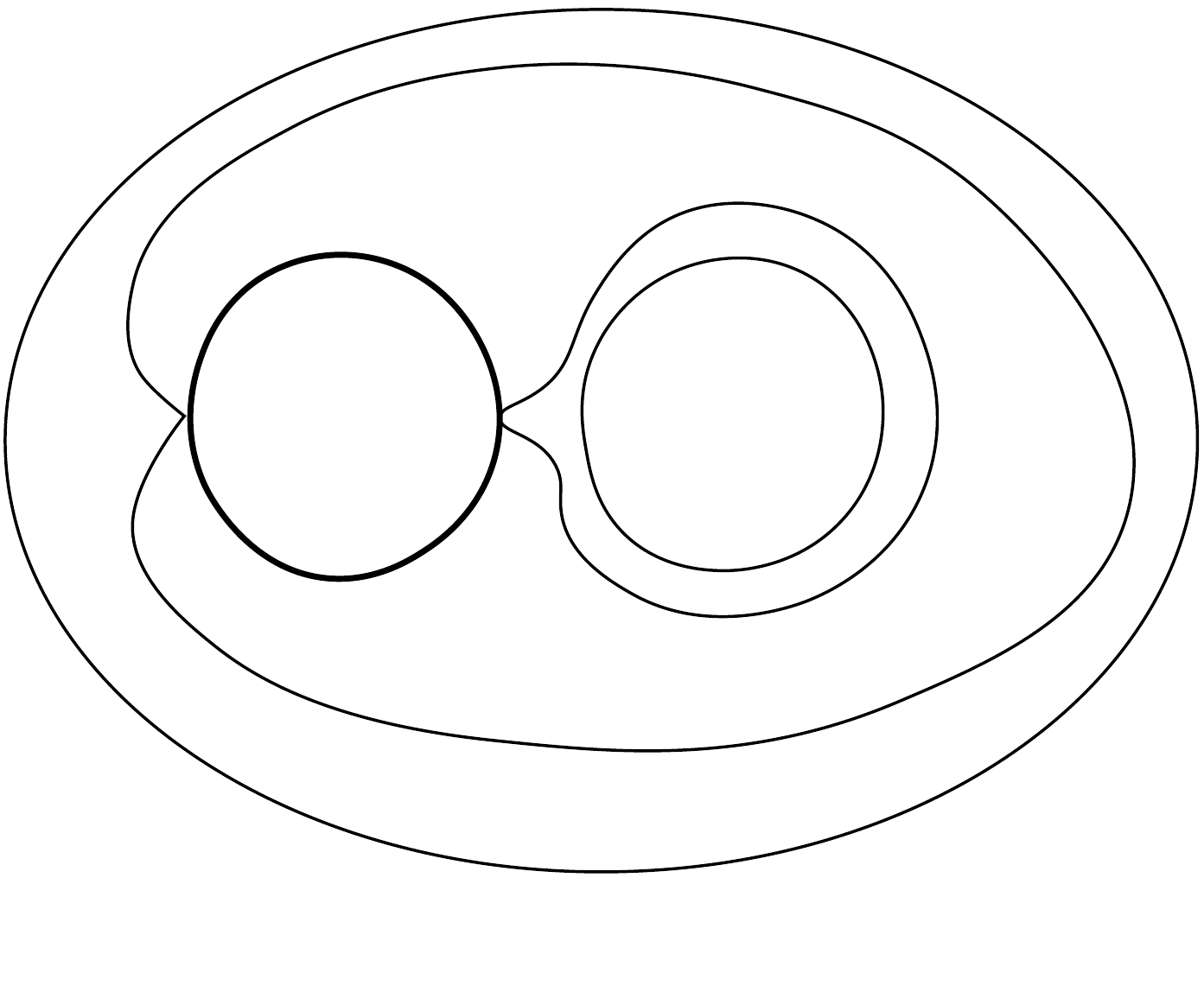,width=9.0cm,angle=0} }}
\vspace{-20pt}
\end{center}
\caption{$b=c=0$, $a>0$.}
\end{figure}

This finishes the proof of the theorem.
\end{proof}

The above two theorems establish another equivalent property for a Riemann surface to be parabolic. Namely,

\begin{thm}
\label{thm:par-ch-cross-cuts}
Let $X=\mathbb{H}/\Gamma$ be a Riemann surface with $\Gamma$ of the first kind. Then $X$ is of parabolic type if and only if the set of horizontal trajectories of each integrable holomorphic quadratic differential on $X$ that are cross-cuts is of zero area.
\end{thm}

\begin{proof}
If $X$ is parabolic, then Theorem \ref{thm:par-int} implies that the set of horizontal cross-cuts is of measure zero. If $X$ is not of parabolic type, then Theorem \ref{thm:int-non-par-cross-cut} gives an integrable holomorphic quadratic differential whose horizontal cross-cut trajectories cover the whole $X$ up to a measure zero. The theorem is proved. 
\end{proof}

The above theorem, together with Theorem \ref{thm:main}, gives the following criteria for finding surfaces that are not parabolic.

\begin{thm}
\label{thm:non-parabolic-criteria}
Let $X=\mathbb{H}/\Gamma$ be a Riemann surface with $\Gamma$ of the first kind. Then $X$ is not of parabolic type if and only if there is a measured lamination $\mu\in {ML}_{\mathrm{int}}(X)$ such that the $\mu$-measure of the geodesics of the support of $\mu$ leaving every compact subset of $X$ is positive.
\end{thm}

We show that ${ML}_{\mathrm{int}}(X)$ is a quasiconformal invariant.

\begin{thm}
\label{thm:mlf-qc-inv}
Let $g:X\to Y$ be a quasiconformal map between two infinite Riemann surfaces whose covering groups are of the first kind, and let  $\hat{g}:G(X)\to G(Y)$ be the induced bijection on the spaces of geodesics. Then $\hat{g}$ induces a bijection  $\hat{g}:ML(X)\to {ML}(Y)$ such that
$$
\hat{g}({ML}_{\mathrm{int}}(X))={ML}_{\mathrm{int}}(Y). 
$$
\end{thm}

\begin{proof}
In \cite{Saric20}, it is established that a homeomorphism $g:X\to Y$ between two Riemann surfaces whose fundamental groups are of the first kind induces a natural bijection $\hat{g}:G(X)\to G(Y)$. The push forward of the measures from $G(X)$ to $G(Y)$ under the bijection $\hat{g}:G(X)\to G(Y)$ induces a bijection between the space of measures on $G(X)$ and $G(Y)$. Since $\hat{g}:G(X)\to G(Y)$ maps support to support, it follows that $\hat{g}$ preserves the spaces of measured laminations. 

It remains to prove that $
\hat{g}({ML}_{\mathrm{int}}(X))={ML}_{\mathrm{int}}(Y)$. A quasiconformal map $g:X\to Y$ is homotopic to a real analytic quasiconformal map $h:X\to Y$ by Douady-Earle \cite{DE}. For $\mu\in {ML}_{\mathrm{int}}(X)$, let $\mathcal{F}_{\mu}$ be a partial foliation of $X$ representing $\mu$. If $\{v_i\}_i$ is the family of real-valued functions defined on a family $\{ U_i\}_i$ of open neighborhoods in $X$ which can pairwise intersect along boundaries, then $\{ v_i\circ h^{-1}\}_i$ and $\{ h(U_i)\}$ form a partial foliation of $Y$ and let $K$ be the quasiconformal constant of $h$. We denote by $g(\mathcal{F}_{\mu})$ the corresponding partial foliation. The Dirichlet integral satisfies (see \cite{Ahlfors})
$$
\mathcal{D}_Y(g(\mathcal{F}_{\mu}))\leq K\mathcal{D}_X(\mathcal{F}_{\mu})<\infty 
$$
and $g(\mathcal{F}_{\mu})$ corresponds to a unique geodesic lamination $g(\mu )\in {ML}_{\mathrm{int}}(Y)$. Since $g$ is invertible, 
we obtain the bijection.
\end{proof}

\section{Approximations by Jenkins-Strebel differentials and Kerckhoff's formula for the Teichm\"uller metric}

In most of this section (except Theorem \ref{thm:ext-l-cont}), the infinite Riemann surface $X=\mathbb{H}/\Gamma$ is assumed to be parabolic, i.e., $X\in O_G$. We first establish the density of the holomorphic quadratic differentials whose non-singular horizontal trajectories make a single cylinder on $X$. We note whether simple closed geodesics with positive weights are dense in the space of bounded measured laminations on $X$ is still unknown.

\begin{definition}
An integrable holomorphic quadratic differential $\varphi$ on $X$ is called a {\it Jenkins-Strebel} differential if its non-singular horizontal trajectories are all closed, homotopic, and make up an open cylinder on $X$. 
\end{definition}

The $|\varphi |$-area of the complement of the cylinder of a Jenkins-Strebel differential is zero (see \cite{Strebel}). For every simple closed geodesic $\gamma$ in $X$ and $b>0$, there is a unique Jenkins-Strebel differential $\varphi_{\gamma}$ with a cylinder homotopic to $\gamma$ and height $b>0$ (see \cite[Theorem 21.1]{Strebel}). We prove

\begin{thm}
\label{thm:approx-JS}
Let $X\in O_G$ be an infinite Riemann surface and $\varphi\in A(X)$. Then there exists a sequence $\{\varphi_n\}_n$ of Jenkins-Strebel differentials on $X$ such that
\begin{enumerate}
\item $\int_X|\varphi_n|\leq\int_X|\varphi |$ for all $n$, and
\item $\varphi_n$ converge to $\varphi$ uniformly on compact subsets of $X$. 
\end{enumerate}
\end{thm}

\begin{proof}
Consider a geodesic pants decomposition of $X$ and an exhaustion $\{X_k\}_k$ of $X$ by finite area surfaces with geodesic boundary obtained by unions of the pairs of pants from the decomposition. We repeat the construction from Section 3.1 to introduce related exhaustion of $X$ such that the horizontal foliation of $\varphi$ is transverse to the boundary curves. For simplicity, denote this exhaustion by $\{X_k\}_k$ as well.

Given $\epsilon >0$, there exists $k_0=k_0(\epsilon )>0$ such that $\int_{X\setminus X_{k_0}} |\varphi |<\epsilon$. Consider the horizontal trajectories of $\varphi$ that intersect $X_{k_0}$. 
Since $X\in O_G$, by Theorem \ref{thm:par-int} the $|\varphi |$-area of the subset of the horizontal trajectories that connect $X_{k_0}$ and $X\setminus X_k$ is going to zero as $k\to \infty$ for any fixed $k_0$. There is a positive lower bound to the $\sqrt{|\varphi |}$-length of horizontal trajectories of $\varphi$ that connects $X_{k_0}$ and  $X\setminus X_k$ because each horizontal trajectory crosses a fixed collar of a boundary component of $X_{k_0}$ (the $|\sqrt{\varphi}|$-distance between boundary components of a collar is positive). Since the $|\varphi |$-area of the set of trajectories connecting $X_{k_0}$ and $X\setminus X_k$ is going to zero, it follows that the transverse measure of the above trajectories is going to zero as $k\to\infty$. Choose $k=k(\epsilon )$ large enough that the area and the transverse measure of the above trajectories are less than $\epsilon$. 

Define $\mathcal{F}_{\epsilon}$ to be the partial foliation consisting of the horizontal trajectories of $\varphi$ intersecting $X_{k_0}$ and not leaving $X_k$, where $k_0=k_0(\epsilon )$ and $k=k(\epsilon )$ are as above. The restriction of $\varphi$ to $X_k$ is an integrable holomorphic quadratic differential. By Strebel \cite[\S 13]{Strebel}, the trajectory structure of $\varphi$ consists of at most countably many maximal cylinders, cross-cut strips, and spiral sets up to a set of the $|\varphi |$-area zero.

Therefore, up to a set of the $|\varphi |$-area zero, any leaf of $\mathcal{F}\cap X_k$ which is not in $\mathcal{F}_{\epsilon}$ is completely contained in $X_k\setminus X_{k_0}$ (belongs to a spiral set) or it connects $\partial X_{k_0}$ with $\partial X_k$ (belongs to a strip of cross-cuts on $X_k\setminus X_{k_0}$) or it connects $\partial X_{k_k}$ to itself (belongs to a strip of cross-cuts on $X_k$). 
The union of the countably many of these sets is measurable. By erasing the union of these strips, $\mathcal{F}_{\epsilon}$ remains a partial measured foliation of $X_k$.

It follows that
$$
D(\mathcal{F}_{\epsilon})\leq\int_{X_k}|\varphi |< \int_X|\varphi |
$$ 
and that the heights $h_{\mathcal{F}_{\epsilon}}(\gamma )$ converge to $h_{\varphi}(\gamma )$ for all simple closed geodesics $\gamma$ in $X$ as $\epsilon\to 0$. 
By doubling argument as in the proof of Theorem 3.3 using Hubbard-Masur \cite{HubbardMasur} and invariance under the reflection in the boundary, there exists an integrable holomorphic quadratic differential $\psi_{\epsilon , k}$ on $X_k$ which realizes $\mathcal{F}_{\epsilon}$ with $\int_{X_k}|\psi_{\epsilon ,k}|\leq D(\mathcal{F}_{\epsilon})\leq\int_{X_k}|\varphi |< \int_X|\varphi |$ for all $\epsilon >0$ and $k=k(\epsilon )$, where the first inequality is proved in \cite[Theorem 24.2]{Strebel}, and it was used in the case of a double Riemann surface in the proof of  Lemma \ref{lem:lim-not-zero}.


The partial foliation $\mathcal{F}_{\epsilon}$ can be straightened to a (hyperbolic) measured (geodesic) lamination $\mu_{\epsilon}$ on $X_k$ whose support does not intersect the boundary of $X_k$. By Penner-Harer \cite[Theorem 3.1.3]{PennerHarer}, there exists a sequence of simple closed geodesics $\{ g_i\}_i$ and weights $\{ w_i\}_i$ such that the weighted simple closed geodesics $w_i\delta_{g_i}$ converge to $\mu_{\epsilon}$ as $i\to\infty$ in the weak* topology on the space of compactly supported measured laminations $ML_0(X_k)$, where $\delta_{g_i}$ is the Dirac measure supported on $g_i$. 

We apply the doubling argument again. Let $\widehat{X}_k$ be the double of $X_k$ and let $\widehat{\mu}_{\epsilon}$ be the measured lamination which equals $\mu_{\epsilon}$ on $X_k$ and its mirror image $\bar{\mu}_{\epsilon}$ on $\widehat{X}_k\setminus X_k$.
The measure laminations $\mu_{\epsilon}$ and $\bar{\mu}_{\epsilon}$ are separated by $\partial X_k$. The mirror image of $g_i$ is denoted by $\bar{g}_i$. By the symmetry, we have that $w_i\delta_{g_i}+w_i\delta_{\bar{g}_i}$ converges to $\widehat{\mu}_{\epsilon}$ as $i\to\infty$ in the weak* topology on $ML(\widehat{X}_k)$. 

By  Hubbard-Masur \cite{HubbardMasur}, there exists an integrable holomorphic quadratic differential $\widehat{\psi}_{i,k}$ on $\widehat{X}_k$ that realizes $w_i\delta_{g_i}+w_i\delta_{\bar{g}_i}$. The non-singular leaves of the differential $\widehat{\psi}_{i,k}$ make exactly two cylinders homotopic to $g_i$ and $\bar{g}_i$. Since $\widehat{\psi}_{i,k}$ is invariant under the mirror symmetry, and $g_i$ and $\bar{g}_i$ are permuted, it follows that the cylinder homotopic to $g_i$ is in $X_k$. 

In particular, the holomorphic quadratic differential $\psi_{i,k}$ obtained by the restriction of $\widehat{\psi}_{i,k}$ to $X_k$ is Jenkins-Strebel.  By the continuity in \cite{HubbardMasur}, we have that $\int_{X_k}|\psi_{i,k}|\to\int_{X_k}|\psi_{\epsilon ,k}|$ as $i\to\infty$. Thus by slightly decreasing $w_i$, we can arrange that $w_i\delta_{g_i}$ converges to $\mu_{\epsilon}$ in the weak* topology,  $\int_{X_k}|\psi_{i,k}|\to\int_{X_k}|\psi_{\epsilon ,k}|$ as $i\to\infty$ and $\int_{X_k}|\psi_{i,k}|\leq \int_{X_k}|\psi_{\epsilon ,k}|$ for all $i$.

By Theorem \ref{thm:main}, there exists an integrable holomorphic quadratic differential $\varphi_{i,k}$ on $X$ that realizes $w_i\delta_{g_i}$ and satisfies $\int_{X}|\varphi_{i,k}|\leq \int_{X_k}|\psi_{\epsilon ,k}|< \int_X|\varphi |$. The first inequality follows holds because $\psi_{i,k}$ is supported on $X_k$ and its cylinder can be thought of as a partial measured foliation on $X$ realizing $w_i\delta_{g_i}$, which gives $\int_X| \varphi_{i,k}|\leq \int_X| \psi_{i,k}|$.

Choose $\epsilon =\frac{1}{n}$ and by the diagonal process, there exist $i=i(n)$ and 
$\varphi_n\in A(X)$ whose horizontal foliation is equivalent to that of $\psi_{i(n),k}\in A(X_k)$ such that the heights of $\varphi_{n}$ are converging to the heights of $\varphi$. In particular, $\varphi_n$ is a Jenkins-Strebel holomorphic quadratic differential on $X$.

In conclusion, since $\int_X |\varphi_n|\leq\int_{X_k}|\psi_{i,k}|\leq\int_X|\varphi|$ there is a subsequence of $\varphi_n$ that converges uniformly on compact subsets to an integrable holomorphic quadratic differential on $X$. Since the limiting quadratic differential heights are identical to the heights of $\varphi$, the Heights Theorem \cite{Saric-heights} gives that the limit is $\varphi$. The theorem is proved.
\end{proof}

Theorem \ref{thm:approx-JS} and Lemma \ref{lem:L^1-conv} immediately give the following corollary.

\begin{cor}
\label{cor:jenkins-dense-L1}
Let $X\in O_G$ be an infinite Riemann surface and $\varphi\in A(X)$. Then there exists a sequence $\{\varphi_n\}_n$ of Jenkins-Strebel differentials on $X$ such that
$$
\int_X|\varphi -\varphi_n|\to 0
$$
as $n\to\infty$.
\end{cor}

We use the existence of the above approximation to prove an analogous formula to Kerckhoff's formula (for the distance in the Teichm\"uller space of closed Riemann surfaces) for the case of parabolic Riemann surfaces. 

Let $C$ be an annular domain in a Riemann surface $X$. Then $C$ is conformal to a Euclidean annulus $\{ z\in\mathbb{C}: 0\leq r_1<|z|<r_2\leq \infty\}$. The {\it modulus} of $C$ is
$$
\mathrm{mod}C=\frac{1}{2\pi}\log\frac{r_2}{r_1}
$$
when $0<r_1<r_2<\infty$ and
$$
\mathrm{mod}C=\infty
$$
when $r_1=0$ and/or $r_2=\infty$. When $C$ is represented by a Euclidean cylinder with the height $h$ and the circumference $l$, then $\mathrm{mod} C=h/l$.

We recall the definition of the extremal length of a simple closed geodesic on $X$.

\begin{definition}
Let $\gamma$ be a simple closed curve on a Riemann surface $X$. The {\it extremal length} $\mathrm{ext}(\gamma )$ of $\gamma$ is 
$$
\mathrm{ext}(\gamma )=\inf_C 1/\mathrm{mod} C,
$$
where $C$ ranges over all annular domains in $X$ homotopic to $\gamma$. 
\end{definition}

For $r>0$, we extend the extremal length to $r\cdot \gamma$ by the formula
$$
\mathrm{ext}(r\cdot\gamma)=r^2\mathrm{ext}(\gamma ).
$$

Kerckhoff observed the following lemma \cite{Kerckhoff}. 

\begin{lem}
Let $\gamma$ be a simple closed geodesic on an infinite Riemann surface $X=\mathbb{H}/\Gamma$ with $\Gamma$ of the first kind and let $r>0$. If the union of all non-singular trajectories of $\varphi\in A(X)$ is a single cylinder in the homotopy class of $\gamma$ of height $r$, then
$$
\mathrm{ext}(r\cdot\gamma )=\int_X |\varphi |.
$$ 
\end{lem}

\begin{proof}
Let $\varphi$ be the holomorphic quadratic differential with a single cylinder in the homotopy class of $\gamma$ whose height is $r$. 
The cylinder $C$ of $\varphi$ has height $r$ by the assumption and denotes its length by $l>0$. Jenkins \cite{Jenkins} proved that the cylinder $C$ of $\varphi$ has the maximal modulus among all annular domains in $X$ (see also Strebel \cite[p 107, Theorem 21.1]{Strebel}). Then we have
$$
\mathrm{ext}(r\cdot\gamma )=r^2\mathrm{ext}(\gamma )=r^2\frac{1}{\mathrm{mod}C}=r^2\frac{1}{(r/l)}=rl=\int_X|\varphi |.
$$
\end{proof}

Using the equivalent definition of $\mathrm{ext}(r\cdot\gamma )$ from the above lemma, we introduce the extremal length for any $\mu\in ML_{\mathrm{int}}(X)$.

\begin{definition}
Let $\mu\in ML_{\mathrm{int}}(X)$ and $\varphi_{\mu}\in ML_{\mathrm{int}}(X)$ the integrable holomorphic quadratic differential whose horizontal foliation realizes $\mu$. The {\it extremal length} of $\mu$ is given by
$$
\mathrm{ext}(\mu )=\int_X |\varphi_{\mu}|. 
$$
\end{definition}

We prove that the extremal length is a continuous function on $ML_{\mathrm{int}}(X)$ without the assumption that $X$ is parabolic.

\begin{thm}
\label{thm:ext-l-cont}
Let $X=\mathbb{H}/\Gamma$ be an infinite Riemann surface with $\Gamma$ of the first kind. Then
the extremal length 
$$
\mathrm{ext}:ML_{\mathrm{int}}(X)\to\mathbb{R}
$$
is a continuous function. 
\end{thm}

\begin{proof}
Recall that the topology on $ML_{\mathrm{int}}(X)$ is induced by identifying it with $A(X)$. The $ L^1$-norm induces a topology on $A(X)$. The continuity is immediate by the definition of the extremal length.
\end{proof}

By Theorem \ref{thm:ext-l-cont} and Corollary \ref{cor:jenkins-dense-L1}, we immediately have

\begin{thm}
Let $X\in O_G$. Then for each $\mu\in ML_{\mathrm{int}}(X)$ there exists a sequence of simple closed geodesics $\{\gamma_n\}_n$ and positive numbers $\{ r_n\}_n$ such that
$$
\mathrm{ext}(r_n\cdot \gamma_n)\to\mathrm{ext}(\mu )
$$
as $n\to\infty$.
\end{thm}

We are ready to prove a formula for the Teichm\"uller distance on the Teichm\"uller space $T(X)$ when $X\in O_G$  which is analogous to the Kerckhoff's formula (see \cite[Theorem 4]{Kerckhoff}) in the case of a closed surface. 

Recall that the Teichm\"uller space $T(X)$ of a Riemann surface $X=\mathbb{H}/\Gamma$ is the set of equivalence classes $[f]$ of quasiconformal maps $f:X\to Y$, where $Y$ is a variable Riemann surface. Two quasiconformal maps $f:X\to Y$ and $f_1:X\to Y_1$ are equivalent if there exist lifts $\tilde{f},\tilde{f}_1:\mathbb{H}\to\mathbb{H}$ of $f,f_1$ and a M\"obius map $A:\mathbb{H}\to\mathbb{H}$ such that $A\circ\tilde{f}=\tilde{f}_1$ on the boundary $\bar{\mathbb{R}}=\mathbb{R}\cup\{\infty\}$ of $\mathbb{H}$ (for example, see \cite[page 117]{GardinerLakic}). The Teichm\"uller distance between two points $[f],[g]\in T(X)$ is given by
$$
d_T([f],[g])=\inf_{q}\frac{1}{2}\log K(q),
$$
where the infimum is over all quasiconformal maps $q$ (boundedly) homotopic to $g\circ f^{-1}$ and $K(q)$ is the quasiconformal constant of $q$ (see \cite[page 25]{GardinerLakic}).

\begin{thm}
\label{thm:kerckhoff}
Let $X\in O_G$ and $S$ be the set of simple closed geodesics on $X$. The Teichm\"uller distance between two points $[f:X\to Y]$ and $[g:X\to Z]$ in $T(X)$ is equal to
$$
d_T([f],[g])=\frac{1}{2}\log \sup_{\gamma\in S}\frac{\mathrm{ext}_Z(g(\gamma ))}{\mathrm{ext}_Y(f(\gamma ))}
$$
where $\mathrm{ext}_Y(\cdot )$ and $\mathrm{ext}_Z(\cdot )$ are extremal lengths on surfaces $Y$ and $Z$ of the corresponding simple closed curves.
\end{thm}

\begin{proof}
Recall that each homotopy class of quasiconformal maps between two Riemann surfaces contains at least one extremal quasiconformal map, i.e., a map with the smallest quasiconformal constant (for example, see \cite[page 26]{GardinerLakic}). Let $h\in [g\circ f^{-1}]$ be an extremal map. Then $d_T([f],[g])=\frac{1}{2}\log K(h)$, where $K(h)$ is the quasiconformal constant of $h$. Recall that if $C$ is any cylinder, then
$$
\frac{1}{K(h)}\mathrm{mod}C\leq\mathrm{mod}h(C)\leq K(h)\mathrm{mod}C.
$$  

Since $\mathrm{ext}(\gamma )=\frac{1}{\mathrm{mod}C_{\gamma}}$ where $C_{\gamma}$ is a cylinder homotopic to $\gamma$ with maximal modulus and since $h(f(\gamma ))$ is homotopic to $g(\gamma )$, it follows that
$$
\mathrm{ext}_Z(g(\gamma ))\leq \frac{1}{\mathrm{mod}h(C_{f(\gamma )})}\leq K(h)\mathrm{ext}_Y(f(\gamma )) 
$$
which gives
\begin{equation}
\label{eq:dT-ineq}
d_T([f],[g])\geq\frac{1}{2}\log \sup_{\gamma\in S}\frac{\mathrm{ext}_Z(g(\gamma ))}{\mathrm{ext}_Y(f(\gamma ))}.
\end{equation}

We need to prove the opposite inequality. Assume that $h$ is given by an affine stretching in the horizontal direction in the natural parameter of an integrable holomorphic quadratic differential $\varphi\in A(Y)$. Such extremal maps are unique in their equivalence classes $[h]$, and they are said to be of {\it Teichm\"uller type}. Let $\mathcal{F}_{\varphi}$ be the horizontal foliation of $\varphi$ and let $\mu_{\varphi}\in ML_{\mathrm{int}}(Y)$ be the corresponding measured lamination. To have the same vertical measure of the foliation $h(\mathcal{F}_{\varphi})$ as the vertical measure of $\mathcal{F}_{\varphi}$, we normalize $h$ to be the horizontal stretching by the factor $K(h)$ while keeping the vertical direction fixed. If $\varphi_h\in A(Z)$ is the terminal differential of the Teichm\"uller map $h:Y\to Z$, then $\int_Z|\varphi_h| =K(h)\int |\varphi |$. By the definition of extremal lengths, we get
$$
\frac{\mathrm{ext}(h(\mu_{\varphi}))}{\mathrm{ext}(\mu_{\varphi})}=K(h).
$$
By Corollary \ref{cor:jenkins-dense-L1} , Theorem \ref{thm:ext-l-cont} and $\mathrm{ext}(r\cdot\gamma )=r^2\mathrm{ext}(\gamma )$, we conclude that there exists a sequence $\{\gamma_n\}$ of simple closed geodesics such that
$
d_T([f],[g])=\lim_{n\to\infty} \frac{1}{2}\log \frac{\mathrm{ext}_Z(g(\gamma_n ))}{\mathrm{ext}_Y(f(\gamma_n ))}$ and the equality holds for $h$ of Teichm\"uler type.

If an extremal map $h\in [g\circ f^{-1}]$ is not of Teichm\"uller type, then it can be approximated by a sequence of maps $h_n:Y\to Z_n$ of the Teichm\"uller type (see \cite[page 106, Theorem 12]{GardinerLakic}). Thus each $h_n$ is associated with an integrable holomorphic quadratic differential $\varphi_n\in A(Y)$ such that $h_n$ is obtained by horizontal stretching by a factor $K(h_n)$ with $K(h_n)\to K(h)$ as $n\to\infty$. 
Let $d_n=d_T([h_n],[h])$.
By the above considerations, for each $n$ there exists a sequence $\{\gamma_i^n\}_i$ of simple closed geodesics such that 
$$
d_T([f],[g_n])=\lim_{i\to\infty}  \frac{1}{2}\log \frac{\mathrm{ext}_{Z_n}(g_n(\gamma_i^n ))}{\mathrm{ext}_Y(f(\gamma_i^n ))}.
$$
Thus there exists a choice $i_n$ for each $n$ such that 
\begin{equation}
\label{eq:ext-distance-approx}
|d_T([f],[g_n])- \frac{1}{2}\log \frac{\mathrm{ext}_{Z_n}(g_n(\gamma_{i_n}^n ))}{\mathrm{ext}_Y(f(\gamma_{i_n}^n ))}|<\frac{1}{n}.
\end{equation} 

For every simple closed geodesic $\gamma$ in $X$ we have 
 $$e^{-2d_n} \mathrm{ext}_{Z}(g(\gamma ))\leq \mathrm{ext}_{Z_n}(g_n(\gamma ))\leq e^{2d_n} \mathrm{ext}_{Z}(g(\gamma )).$$
 and together with (\ref{eq:ext-distance-approx}) we conclude that
 \begin{equation}
\label{eq:ext-distance}
|d_T([f],[g_n])- \frac{1}{2}\log \frac{\mathrm{ext}_{Z}(g(\gamma_{i_n}^n ))}{\mathrm{ext}_Y(f(\gamma_{i_n}^n ))}|<\frac{1}{n}+d_n.
\end{equation} 
We obtained the theorem by letting $n\to\infty$ and noting that (\ref{eq:dT-ineq}) holds.
\end{proof}

\section{Surfaces with a bounded pants decomposition}

In this section $X=\mathbb{H}/\Gamma$ is an infinite Riemann surface with the fundamental group $\Gamma$ of the first kind that   has a geodesic pants decomposition $\{ P_k\}_k$ with boundary geodesics (cuffs) $\{\alpha_n\}_n$ that satisfy
$$
\frac{1}{C}\leq \ell_X(\alpha_n)\leq C
$$
for a fixed $C>0$ and all $n$, where $\ell_X(\alpha_n)$ is the hyperbolic length of the closed geodesic $\alpha_n$. The Riemann surface $X$ is said to have {\it bounded} pants decomposition. Each $\alpha_n$ is either on the common boundary of two geodesic pairs of pants $P_1$ and $P_2$ or a single pair of pants $P$ glued to itself along $\alpha_n$. In the former case, let $\beta_n$ be the shortest closed geodesic in $P_1\cup\alpha_n\cup P_2$ that intersect $\alpha_n$ in two points. In the latter case, let $\beta_n$ be the shortest closed geodesic in $P\cup\alpha_n$ that intersects $\alpha_n$ in one point.

Let $\varphi$ be an integrable holomorphic quadratic differential on $X$ and let $\mu_{\varphi}\in {ML}_{\mathrm{int}}(X)$ be the measured lamination corresponding to the vertical foliation of $\varphi$. We prove

\begin{prop}
\label{prop:int-numbers}
Let $X$ be an infinite Riemann surface with a bounded pants decomposition $\{\alpha_n\}_n$ and transverse geodesic family $\{\beta_n\}_n$ defined above. For any $\varphi\in A(X)$, we have
\begin{equation}
\label{eq:sum-intersections}
\sum_n [i(\alpha_n,\mu_{\varphi})^2+i(\beta_n,\mu_{\varphi})^2]<\infty .
\end{equation}
\end{prop}

\begin{proof}
Let $\mathcal{C}_n$ be the standard collar around $\alpha_n$. By the Collar Lemma \cite{Buser}, the collars $\mathcal{C}_n$ are mutually disjoint. Let $\mathcal{B}_n$ be the sub-leaves of the horizontal foliation $\mathcal{F}_{\varphi}$ that connect the two boundary components of $\mathcal{C}_n$. Note that by \cite[Lemma 6.1]{Saric-heights}  we have
$$
i(\alpha_n,\mu_{\varphi})\leq i(\alpha_n,\mathcal{B}_n).
$$

Consider a countable collection of vertical arcs that intersect only the leaves of $\mathcal{B}_n$.  Reduce the size of the elements of the collection such that each leaf of $\mathcal{B}_n$ is intersected exactly once. Let $I_n$ be the union of these arcs. 
Let $w:=u+iv=\int_{*}\sqrt{\varphi (z)}dz$ be the natural parameter of $\varphi$. Then we have 
$$i(\alpha_n,\mathcal{B}_n)=\int_{I_n}|dv|$$
by the definition of the intersection number.

The Cauchy-Schwarz inequality and the above give
\begin{equation}
\label{eq:intersection_cuff}
i(\alpha_n,\mu_{\varphi})^2\leq \Big{(}\int_{I_n}dv\Big{)}^2\leq\int_{I_n}l_n(w)dv\int_{I_n}\frac{1}{l_n(w)}dv
\end{equation}
where $l_n(w)$ is the length of the leaf of $\mathcal{B}_n$ through $w$ in the $\varphi$-metric. 

Note that by the Fubini theorem $$\int_{I_n}l_n(w)dv=\int_{\widehat{\mathcal{B}}_n}dudv\leq\int_{\mathcal{C}_n}|\varphi (z)|dxdy,$$ where $\widehat{\mathcal{B}}_n$ stands for the region in $\mathcal{C}_n$ which is the union of the leaves of $\mathcal{B}_n$. 
Moreover, in the natural parameter the leaves of $\mathcal{B}_n$ are horizontal arcs and we obtain (see \cite{Saric-heights} or \cite{BHS})
$$
\int_{I_n}\frac{1}{l_n(w)}dv=\mathrm{mod}\mathcal{B}_n\leq\mathrm{mod} \mathcal{C}_n\leq C'
$$
where $\mathrm{mod}\mathcal{B}_n$ is the modulus of the family of curves $\mathcal{B}_n$ and $\mathrm{mod} \mathcal{C}_n$ is the modulus of all curves in $\mathcal{C}_n$ that connect the two boundary components of $\mathcal{C}_n$. The inequality $\mathrm{mod} \mathcal{C}_n\leq C'$ holds for all $n$ by the lower bound on the lengths $\ell_X(\alpha_n)$ (see \cite{Maskit}, \cite{BHS} or \cite{Saric-heights} for details).

Therefore from (\ref{eq:intersection_cuff}) we obtain
$$
i(\alpha_n,\mu_{\varphi})^2\leq C'\int_{\mathcal{C}_n}|\varphi (z)|dxdy
$$
Since $\mathcal{C}_n$ are disjoint, summing over all $n$ gives
\begin{equation}
\label{eq:int-sum-cuffs}
\sum_n i(\alpha_n,\mu_{\varphi})^2\leq C'\|\varphi\|_{L^1(X)}.
\end{equation}

Consider the family of simple closed geodesics $\{\beta_n\}_n$ associated to $\{\alpha_n\}_n$ as above. Note that $\ell_X(\beta_n)$ is also bounded between two positive constants because $\ell_X(\alpha_n)$ is bounded, and $\beta_n$ is chosen to be of minimal length among all closed geodesics intersecting $\alpha_n$ in a minimal number of points inside $P_1\cup\alpha_n\cup P_2$ or $P\cup\alpha$. In addition, each $\beta_i$ can intersect at most four other geodesics from the family $\{\beta_n\}_n$. Since the standard collars around disjoint closed geodesics are disjoint,  each standard collar around $\beta_n$ can intersect at most four standard collars around other geodesics in $\{\beta_n\}_n$. The above argument gives
$$
\sum_n i(\beta_n,\mu_{\varphi})^2\leq 4C'\|\varphi\|_{L^1(X)}
$$
and the proposition is proved.
\end{proof}

We apply Theorem \ref{thm:par-ch-cross-cuts} together with the necessary condition in Proposition \ref{prop:int-numbers} to detect infinite Riemann surfaces with bounded pants decomposition that are parabolic type. 

We first introduce a quantity that will measure the {\it topological complexity with respect to the pants decomposition} of the ends of $X$. Let $X_1$ be the closure of a fixed pair of pants from the pants decomposition. Then $X_1$ is either a pair of pants with at least one boundary geodesic (and at most two cusps) or a torus with a single geodesic boundary component. Since $X$ is an infinite surface, the complement of $X_1$ in $X$ has at least one component with a non-simple topological end (an end accumulated by other ends or genus). Each component of $X\setminus X_1$ has at least one boundary geodesic in common with $X_1$. We define $q(1)$ as the number of the boundary geodesics of $X_1$ that are also on the boundary of a component of $X\setminus X_1$ with a non-simple end. Note that a single component of $X\setminus X_1$ with a non-simple end can have one, two, or three boundary geodesics of $X_1$ on its boundary. In any case, we have $1\leq q(1)\leq 3$.

Assume now that we defined the subsurface $X_n$ and $q(n)$. We define $X_{n+1}$ as the union of $X_n$ and all pairs of pants sharing the common boundary geodesics with $X_n$. Define $q(n+1)$ to be the number of geodesics on the boundary of $X_{n+1}$ that are also boundary geodesics of the components of $X\setminus X_{n+1}$ with non-simple ends. The sequence of numbers $\{ q(n)\}_n$ is called the {\it topological complexity of the ends of} $X$ (with respect to the fixed pants decomposition). 

For example, when the Riemann surface $X$ is the complement of a Cantor set in the Riemann sphere, then a standard pants decomposition as in \S 8 has the topological complexity $q(n)=2^n$. When the cuff lengths of the pants decomposition of the complement $X$ of a Cantor set are fixed, Theorem \ref{thm:bdd-parbolic} below does not apply as $q(n)$ is too large. Indeed, Theorem \ref{thm:bdd-parbolic} is primarily interesting when the topological complexity is not too big.

\begin{thm}
\label{thm:bdd-parbolic}
Let $X$ be an infinite Riemann surface with a bounded pants decomposition. Let $\{ X_n\}_n$ be an exhaustion of $X$ obtained from the pants decomposition and $\{ q(n)\}_n$ the topological complexity of the ends of $X$. 

If 
$$
\sum_{n=1}^{\infty}\frac{1}{q(n)}=\infty .
$$
then the Riemann surface $X$ is of parabolic type.
\end{thm}

\begin{proof}
Assume that 
$$
\sum_{n=1}^{\infty}\frac{1}{q(n)}=\infty .
$$
Suppose that $\mu\in{ML}(X)$ and the set of geodesics in the support of $\mu$ whose both ends leave every compact subset of $X$ have positive $\mu$-measure. Assume that $X_{n_0}$ is the first subsurface such that the $
\mu$-measure of the geodesics of the support of $\mu$ intersecting $X_{n_0}$ is positive. Denote this set of geodesics by $\Lambda_{n_0}$ and  set $M=\mu (\Lambda_{n_0})>0$.  

Each geodesic of $\Lambda_{n_0}$ intersects at least one boundary geodesic of $X_n$ that bounds components of $X\setminus X_n$ with non-simple ends for all $n\geq n_0$. Therefore the $\mu$-measure $M$ distributes to $q(n)$ boundary geodesics of $X_n$. If $\{\alpha_1,\ldots ,\alpha_{q(n)}\}$ are the boundary geodesics of $X_n$ corresponding to non-simple ends of $X\setminus X_n$, then we have
$$
\sum_{i=1}^{q(n)}i(\alpha_i,\mu |_{\Lambda_{n_0}})\geq M.
$$
By the Cauchy-Schwarz inequality, we  have that
$$
\sum_{i=1}^{q(n)}i(\alpha_i,\mu |_{\Lambda_{n_0}})^2\geq \frac{M^2}{q(n)}.
$$
Summing over all $n$ in the above inequality, we obtain $\sum_n i(\alpha_n,\mu )^2=\infty$. This implies that $\mu\notin {ML}_{\mathrm{int}}(X)$. Therefore, by Proposition \ref{prop:int-numbers}, $X$ does not support an integrable holomorphic quadratic differential with a set of horizontal leaves of the positive area that leave each compact subset of $X$. Thus $X$ is of a  parabolic type by Theorem \ref{thm:int-non-par-cross-cut}. 
\end{proof}

We compare the above theorem to the known results for the surfaces with bounded pants decomposition. Mori \cite{Mori} (see also Rees \cite{Rees}) proved that $\mathbb{Z}^2$-covers of a compact surface are of parabolic type while the $\mathbb{Z}^r$-covers are not parabolic when $r\geq 3$. When the $\mathbb{Z}^2$-cover $X$ is obtained by lifting two disjoint simple closed curves,
then we have $q(n)=4n$ and by Theorem \ref{thm:bdd-parbolic}, $X$ is parabolic. Theorem \ref{thm:bdd-parbolic} implies that $X$ is parabolic for $q(n)\leq const\cdot n(\ln n)^p$ for $p\leq 1$ which can be thought of as an interpolation between $\mathbb{Z}^2$- and $\mathbb{Z}^3$-covers.

\section{Surfaces with sequences of closed geodesics whose lengths converge to zero}

In this section, $X$ is an infinite Riemann surface equipped with a geodesic pants decomposition whose geodesic boundaries $\{\alpha_n\}_n$ satisfy
$$
\ell_X(\alpha_n)\leq C
$$
for some $C>0$ and all $n$, and there is a subsequence $\{\alpha_{n_k}\}_k$ with
$$
\lim_{k\to\infty}\ell_X(\alpha_{n_k})=0
$$
and
$$
c\leq\ell_X(\alpha_n)
$$
for some $c>0$ and all $n\notin\{ n_k\}_k$. 

We prove
\begin{thm}
\label{thm:upper-bdd}
Let $X$ be an infinite Riemann surface equipped with a geodesic pants decomposition $\{\alpha_n\}_n$ that satisfies the above conditions. For $\varphi\in A(X)$, we have
\begin{equation}
\label{eq:int-upper-bdd}
\sum_{n=1}^{\infty} \frac{[i(\alpha_n,\mu_{\varphi})]^2}{\ell_X(\alpha_n)}<\infty .
\end{equation}
\end{thm}

\begin{proof}
Consider the standard collar $\mathcal{C}_n$ around $\alpha_n$. By (\ref{eq:intersection_cuff}), we have that
$$
i(\nu_{\varphi},\alpha_n)^2\leq \Big{(}\int_{\mathcal{C}_n}|\varphi (z)|dxdy\Big{)}\cdot \mathrm{mod}\mathcal{B}_n
$$
where $\mathcal{B}_n$ is the family of horizontal arcs of $\varphi$ connecting the two boundaries of the standard collar $\mathcal{C}_n$. Note that $\mathrm{mod}\mathcal{B}_n\leq\mathrm{mod}\mathcal{C}_n$, where $\mathrm{mod}\mathcal{C}_n$ is the modulus of the curve family connecting the boundaries of the standard collar $\mathcal{C}_n$ (see \cite{Ahlfors}). Since $\mathrm{mod}\mathcal{C}_n\leq M\ell_X(\alpha_n)$ for some constant $M>0$ (see \cite{Maskit}), we obtain
$$
\frac{i(\nu_{\varphi},\alpha_n)^2}{\ell_X(\alpha_n)}\leq M\Big{(}\int_{\mathcal{C}_n}|\varphi (z)|dxdy\Big{)}.
$$
By adding over all $n$ we obtain (\ref{eq:int-upper-bdd}) because $\varphi$ is integrable.
\end{proof}

\section{The complement of the Cantor set surface} Let $X$ be a Riemann surface that is conformal to the complement of the Cantor set in the Riemann sphere. We realize the hyperbolic metric on $X$ by gluing geodesic pants along their boundaries by isometries while the twists are arbitrary. At the level $n=1$, we glue two geodesic pair of pants along one boundary component to obtain a surface $X_1$ of genus zero with four disks removed. At the level $n=2$, we glue one geodesic pair of pants to each boundary geodesic of $X_1$ and obtain a surface $X_2$ with zero genus and eight boundary geodesics. We continue this process for all $n$ such that the surface $X_n$ has zero genus and $2^{n+1}$ boundary components. The family $\{ X_n\}_{n=1}^{\infty}$ is exhaustion of $X$ by compact geodesics subsurfaces, and we fix the geodesic pants decomposition of $X$ obtained by the above process. At the level $n$, the geodesic boundaries of $X_n$ are enumerated by $\{ \alpha^j_n\}_{j=1}^{2^{n+1}}$. 

We have the following theorem (see \cite[Theorem 10.3]{BHS}).

\begin{thm}[Basmajian-Hakobyan-\v S]
\label{thm:BHS}
Under the above notation, let $\ell_n^j$ be the length of the geodesic $\alpha^j_n$ on the Riemann surface $X$.
If, for all $1\leq j\leq 2^{n+1}$ and $n\geq 1$,
$$
\ell_n^j\leq\frac{n}{2^{n+1}}
$$
then $X$ is parabolic.
\end{thm}

\begin{proof}
We give a new proof independent of \cite{BHS}. By Theorem \ref{thm:par-ch-cross-cuts}, $X$ is of parabolic type if and only if for any integrable holomorphic quadratic differential $\varphi$ the set of horizontal trajectories that leave every compact subset is of zero $|\varphi |$-measure. Consider a set $T_{n_0}$ of horizontal trajectories of $\varphi$ that leave every compact subset of $X$ and intersect a subsurface $X_{n_0}$. Let $m_{n_0}>0$ be the $|\varphi |$-measure of this set. Then each trajectory of $T_{n_0}$ intersects at least one boundary geodesic $\{ \alpha^j_n\}_{j=1}^{2^{n+1}}$ for each $n\geq n_0$. Therefore the $|\varphi |$-measure $m_{n_0}$ distributes over $\{ \alpha^j_n\}_{j=1}^{2^{n+1}}$ for each $n\geq n_0$. Equivalently, for each $n\geq n_0$, we have
$$
\sum_{j=1}^{2^{n+1}}i(\alpha^j_n,\mu_{\varphi})\geq m_{n_0}.
$$
This gives 
$$
\sum_{j=1}^{2^n}\frac{i(\alpha^j_n,\mu_{\varphi})^2}{\ell^j_n}\geq 2^{(n+1)}\frac{m_{n_0}^2}{2^{2(n+1)}}\frac{1}{\frac{n}{2^{n+1}}}=\frac{m_{n_0}^2}{n}.
$$
Consequently
$$
\sum_{n\geq n_0}\sum_{j=1}^{2^{n+1}}\frac{i(\alpha^j_n,\mu_{\varphi})^2}{\ell^j_n}=\infty .
$$
which contradicts Theorem \ref{thm:upper-bdd}. Thus $m_{n_0}=0$ for all $n_0$ and the $|\varphi |$-measure of escaping horizontal trajectories is zero.
Since $\varphi$ is an arbitrary integrable holomorphic quadratic differential, we conclude that $X$ is parabolic.
\end{proof}

We prove a lemma that will help us in providing a converse to the above theorem.

\begin{lem}
\label{lem:trapezoid}
Let $Q$ be a right-angled trapezoid in the complex plane with bases $a$ and $b$ and the height $h$ that is also one side of $Q$. Assume that $1\leq |b|/|a|\leq C$ and $|h|/|b|\geq C$ for some $C\geq 1$, where $|a|$, $|b|$ and $|h|$ are the Euclidean lengths of the sides.

Consider a foliation of $Q$ whose leaves are Euclidean straight lines connecting $a$ and $b$, and the correspondence between $a$ and $b$ induced by the leaves is a linear (stretch) map for the Euclidean metric.
Then there exists a differentiable function $v:Q\to\mathbb{R}$ such that $v^{-1}(const)$ are the above Euclidean arcs and
$$
\iint|\nabla v|^2dxdy\leq M\cdot |h|\cdot |a|,
$$
where $M=M(C)$ depends only on $C$.
\end{lem}

\begin{proof}
Place $Q\subset\mathbb{C}$ such that the height $h$ is the segment $[0,|h|]\subset \mathbb{R}$, the base $a$ is the segment $[0,i|a|]$ of the $y$-axis and the base $b$ is the vertical segment connecting $|h|$ to $|h|+i|b|$.  
Let $x+iy\in Q$. To define the desired function, we first find the segment through $x+iy$ such that its endpoint on $a$
is $is|a|$ and its endpoint on $b$ is $|h|+is|b|$, where $s\in [0,1]$ (see Figure 6). We need to find $s$ and note that this choice guarantees that the correspondence between the bases $a$ and $b$ under the foliation leaves is linear.

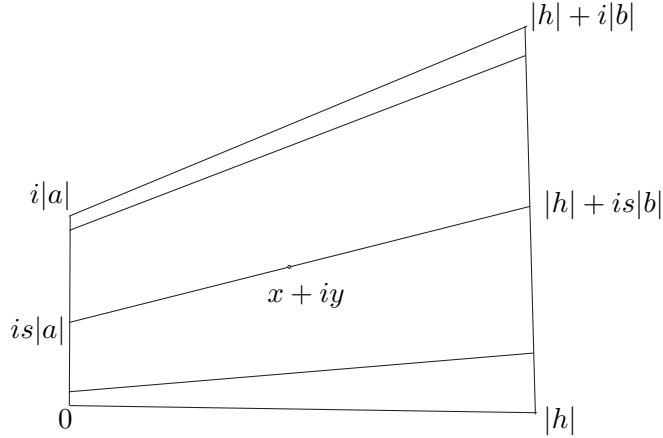
\begin{figure}[h]
\leavevmode 
\SetLabels
\L(.2*.25) $is|a|$\\
\L(.47*.33) $x+iy$\\
\L(.76*.53) $|h|+is|b|$\\
\L(.25*.05) $0$\\
\L(.76*.05) $|h|$\\
\L(.22*.55) $i|a|$\\
\L(.745*.95) $|h|+i|b|$\\
\endSetLabels
\begin{center}
\AffixLabels{\centerline{\epsfig{file =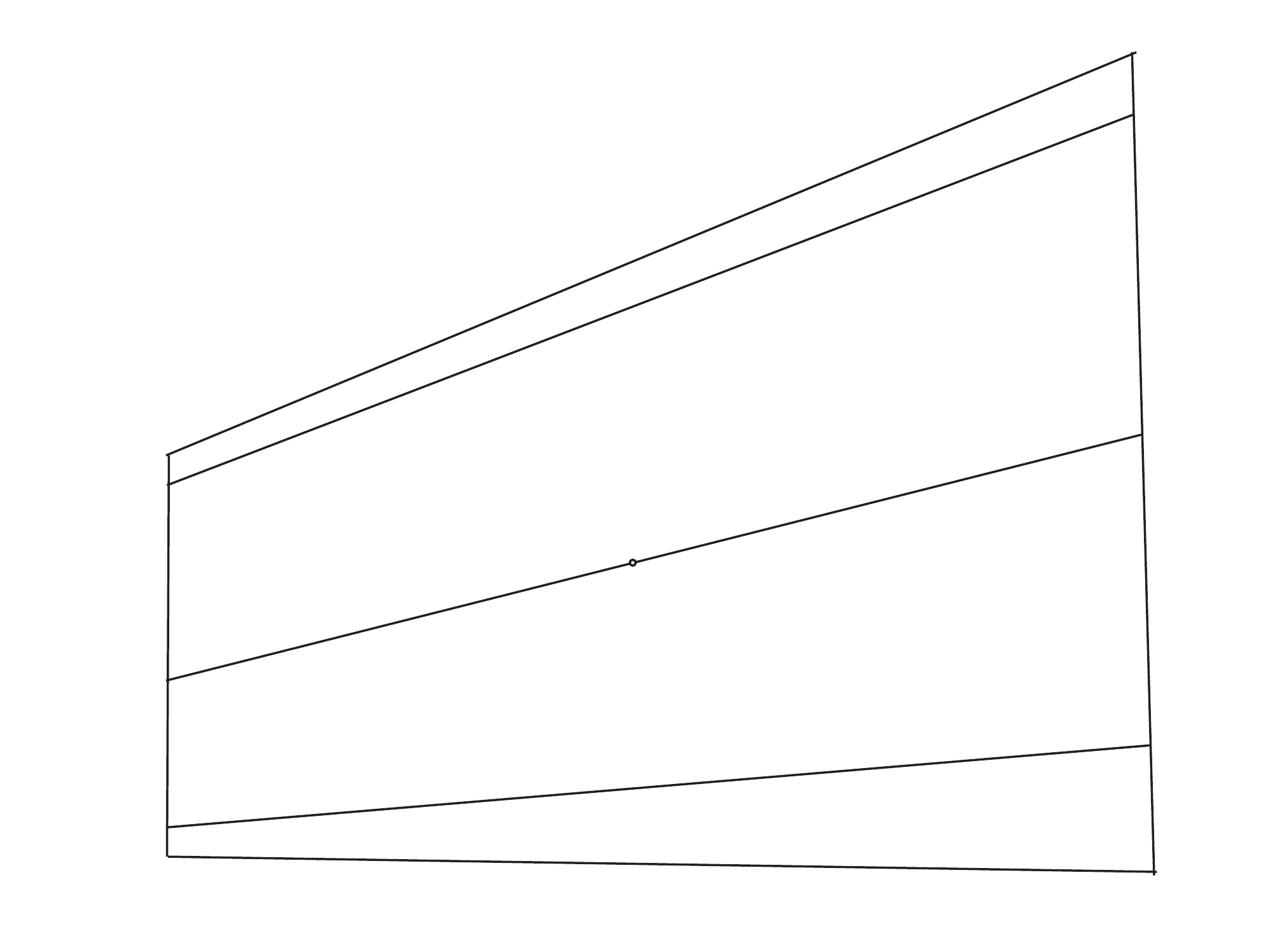,width=8.0cm,angle=0} }}
\vspace{-20pt}
\end{center}
\caption{The foliation in the trapezoid.} 
\end{figure}

To find $s$, note that the vector $x+iy$ is a convex combination of $is|a|$ and $|h|+is|b|$ (see Figure 6). Therefore there exists $t\in [0,1]$ such that
$$
x+iy=its|a|+(1-t)(|h|+is|b|).
$$ 
From the real part of the above equation, we get
$$
t=1-\frac{x}{|h|}.
$$
Combining this with the imaginary part, we obtain
$$
y=(1-\frac{x}{|h|})s|a|+\frac{x}{|h|}s|b|
$$
and solving for $s$ we obtain
$$
s=\frac{y}{(1-\frac{x}{|h|})|a|+\frac{x}{|h|}|b|}.
$$
We define $v:Q\to\mathbb{R}$ by setting it to be constant along the above segments by
$$
v(x+iy)=s|a|
$$ 
where $s$ is a function of $x+iy$. By the above considerations, we get
\begin{equation}
\label{eq:fol-trapezoid}
v(x+iy)=\frac{|a|y}{(1-\frac{x}{|h|})|a|+\frac{x}{|h|}|b|}.
\end{equation}

A direct computation from (\ref{eq:fol-trapezoid}) and an estimation gives
$$
|\nabla v|^2\leq M_1(C)=\frac{(C-1)^2}{C^2}+1.
$$
Therefore
$$
\iint_Q|\nabla v|^2dxdy\leq M\cdot |h|\cdot |a|
$$
for $M=M(C)=M_1C$.
\end{proof}

We can complement the result in Theorem \ref{thm:BHS} by finding a condition that guarantees the Riemann surface $X$ is not parabolic. This is the main result in this section, as methods in \cite{BHS} do not allow us to detect non-parabolic Riemann surfaces unless they contain a funnel or a half-plane. In the case of the geodesic pants decomposition, which is upper bounded, the surface does not have half-planes (and funnels by our definition). 

\begin{thm}
\label{thm:cantor}
Let $X$ be the complement of the Cantor set on the Riemann sphere and $\{ X_n\}_n$ the compact exhaustion as above. Let $\ell_n^j$ be the length of the geodesic $\alpha^j_n$ on the Riemann surface $X$.
If there exists $r>2$ such that, for all $1\leq j\leq 2^{n+1}$ and $n\geq 1$,
$$
\ell_n^j= \frac{(n+1)^r}{2^{n+1}}
$$
then $X$ is not parabolic.
\end{thm}

\begin{proof}
We construct a proper integrable partial foliation $\mathcal{F}$ on $X$. Since being parabolic is a quasiconformal invariant, by \cite[Theorem 1.1]{ALPS}, it is enough to consider $X$ with zero twists. Since $X$ has zero twists, the orthogeodesics between different boundary geodesics (cuffs) of the geodesic pants decomposition connect at the boundary geodesics to make geodesics on $X$ that escape to infinity at both endpoints. These infinite geodesics $\{g_k\}_k$ are mutually disjoint and orthogonal to each boundary geodesic (cuff) of the geodesic pants decomposition they intersect. We form a  partial foliation of $X$ with the help of the family $\{ g_k\}_k$ (see Figure 7). 

\begin{figure}[h]
\leavevmode 
\SetLabels
\L(.66*.9) $g_1$\\
\L(.52*.8) $g_2$\\
\L(.32*.72) $g_3$\\
\L(.65*.72) $g_4$\\
\endSetLabels
\begin{center}
\AffixLabels{\centerline{\epsfig{file =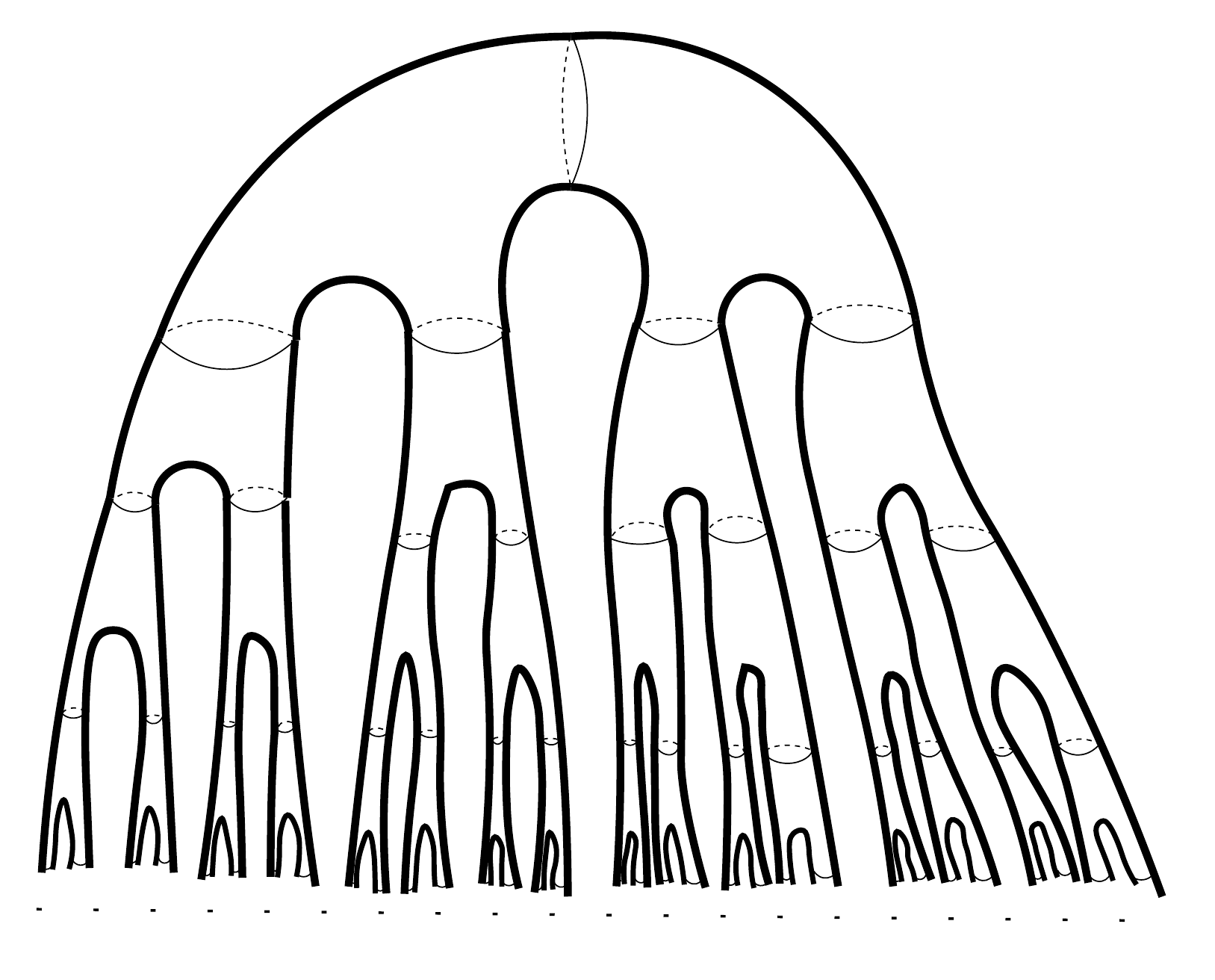,width=12.0cm,angle=0} }}
\vspace{-20pt}
\end{center}
\caption{The bold lines form family $\{ g_k\}_k$.} 
\end{figure}
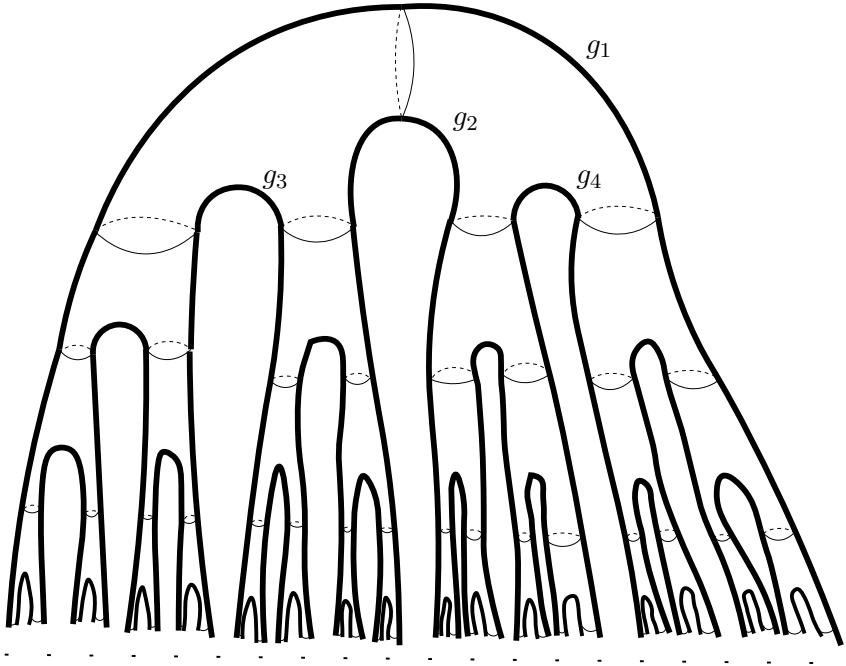

Let $P$ be a pair of pants at the level $n\geq 2$ in the pants decomposition of $X$. Denote by $\alpha_1$ the boundary geodesic curve of $P$ that is in the interior of $X_{n}$ and by $\alpha_2$, $\alpha_3$ the boundary geodesic curves of $P$ that are also on the boundary of $X_{n}$. Then we have, for $n\geq 2$, $$\ell_X(\alpha_1)= \frac{n^r}{2^{n}},$$
 and $$\ell_X(\alpha_2)=\ell_X(\alpha_3)= \frac{(n+1)^r}{2^{n+1}}.$$ 
 This determines the lengths of all cuffs in Figure 7 except the top cuff. We set its length to be $1$.

Let $o_{i,j}$ be the orthogeodesic between $\alpha_i$ and $\alpha_j$, for $i\neq j$ and $i,j\in\{ 1,2,3\}$. Then $o_{1,2}\cup o_{1,3}\cup o_{2,3}$ divides $P$ into two right angled hexagons $\Sigma_1$ and $\Sigma_2$. 
The sides $o_{1,2}$ and $o_{1,3}$ have equal lengths $\ell_X(o_{1,2})=\ell_X(o_{1,3})$. In each hexagon $\Sigma_k$, draw the orthogeodesic $a_{i,k}$ between the half of $\alpha_1$ that is on the boundary of $\Sigma_k$ and $o_{2,3}$. The orthogeodesic $a_{i,k}$ divides the right-angled hexagon into two right-angled pentagons (see Figure 8).  
To shorten the notation, let $\ell_n=
\frac{(n+1)^r}{2^{n+1}}$. Note that  $\ell_n\to 0$ and $\ell_{n}/\ell_{n+1}\to {2}$ as $n\to\infty$, and that $\ell_n$ is decreasing in $n$. If $A$ and $B$ are two quantities that go to infinity, then 
$A\asymp B$ means that $A/B$ is between two positive constants. 

\begin{figure}[h]
\leavevmode 
\SetLabels
\L(.53*.9) $\alpha_1$\\
\L(.54*.5) $a_{i,k}$\\
\L(.33*.6) $o_{1,2}$\\
\L(.67*.6) $o_{1,3}$\\
\L(.5*.1) $o_{2,3}$\\
\L(.3*.15) $\alpha_2$\\
\L(.72*.15) $\alpha_3$\\
\L(.22*.2) $A$\\
\L(.3*.45) $C$\\
\L(.43*.92) $B$\\
\endSetLabels
\begin{center}
\AffixLabels{\centerline{\epsfig{file =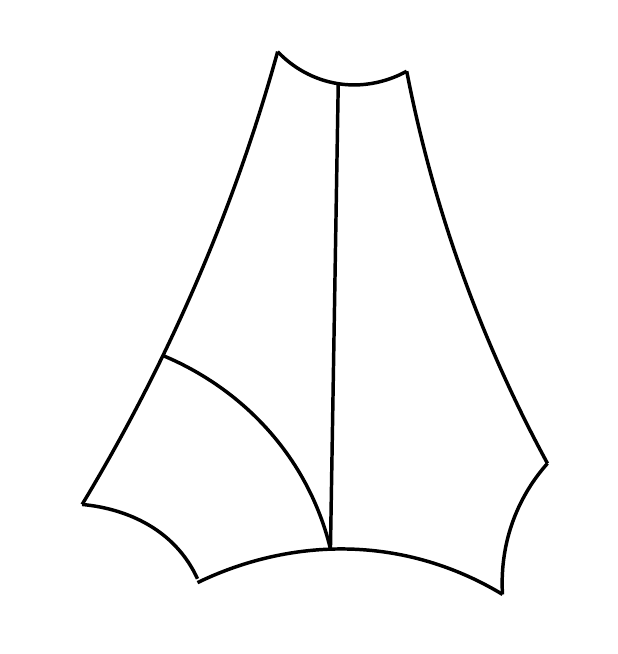,width=8.0cm,angle=0} }}
\vspace{-20pt}
\end{center}
\caption{The hexagon $\Sigma_k$ and the two pentagons obtained by the orthogonal $a_{i,k}$.} 
\end{figure}
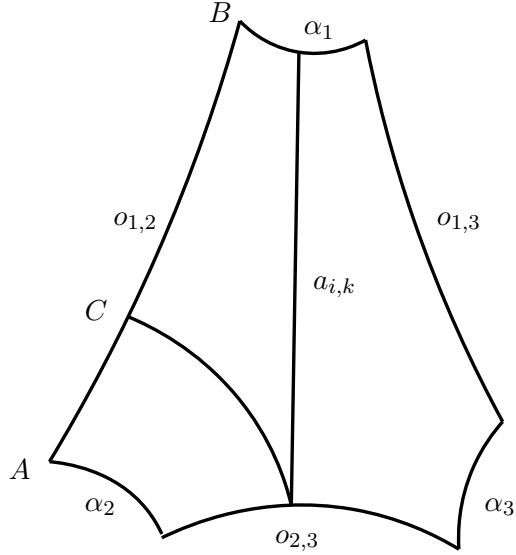

The hyperbolic trigonometry identities in the hexagon (see Figure 8) give
\begin{equation}
\label{eq:pent-eqns}
\begin{array}l
\cosh \ell_X(o_{1,2})=\frac{1}{\tanh\frac{\ell_{n}}{4}\tanh\frac{\ell_{n+1}}{2}},\\
\sinh \ell_X(o_{1,2})\sinh\frac{\ell_{n}}{4}=\cosh \ell_X(a_{i,k}),\\
\cosh \frac{\ell_X(o_{2,3})}{2}=\frac{1}{\tanh \ell_{X}(a_{i,k})\tanh\frac{\ell_{n+1}}{2}}.
\end{array}
\end{equation}
Since $\ell_{n+1}/\ell_n\to 1/2$ as $n\to\infty$, by the equations (\ref{eq:pent-eqns}) we get
\begin{equation}
\label{eq:asympt-lengths}
\begin{array}l
\ell_X(o_{1,2})=\ell_X(o_{1,3})\asymp \log\frac{1}{\ell_n},\\
\ell_X(o_{2,3})\asymp \log\frac{1}{\ell_n},\\
\ell_X(a_{i,k})\asymp \log\frac{1}{\ell_n}.
\end{array}
\end{equation}

We map the pentagon $\Omega_{i,k}$ on the left side of Figure 8 by a diffeomorphism $f$ to a subset of $\mathbb{C}$ as follows. First map the geodesic arc $o_{1,2}$ by an isometry to the interval $[0,\ell_X(o_{1,2})]$ of the real axis.  Each point $w$ in the pentagon belongs to a unique hyperbolic geodesic arc $\gamma_{w_0}$  orthogonal to $o_{1,2}$ with foot $w_0$ on $o_{1,2}$. Map the geodesic arc $\gamma_{w_0}$ to a Euclidean segment orthogonal to $[0,\ell_X(o_{1,2})]$ by an isometry. This defines the map $f$ from the pentagon to a domain in $\mathbb{C}$ which is an isometry on the sides $o_{1,2}$, $\alpha_1$ and $\alpha_2$, and it is a diffeomorphism on the whole pentagon (see Figure 9). The parts of the boundary $f(o_{2,3})$ and $f(a_{i,k})$ are not Euclidean segments. The distance of each point on $a_{i,k}$ to the side $o_{1,2}$ is greater than the length of the side on $\alpha_1$ and the distance of each point on the side $o_{2,3}$ to the side $o_{1,2}$ is greater than the length of the side on $\alpha_2$. Therefore the distance of the point $a_{i,k}\cap o_{2,3}$  to the side $o_{1,2}$ is greater than both sides on $\alpha_1$ and on $\alpha_2$. 

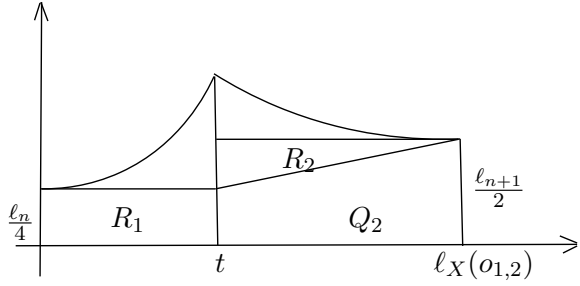
\begin{figure}[h]
\leavevmode 
\SetLabels
\L(.19*.19) $\frac{\ell_n}{4}$\\
\L(.68*.3) $\frac{\ell_{n+1}}{2}$\\
\L(.41*.05) $t$\\
\L(.64*.05) $\ell_X(o_{1,2})$\\
\L(.3*.2) $R_1$\\
\L(.48*.4) $R_2$\\
\L(.55*.2) $Q_2$\\
\endSetLabels
\begin{center}
\AffixLabels{\centerline{\epsfig{file =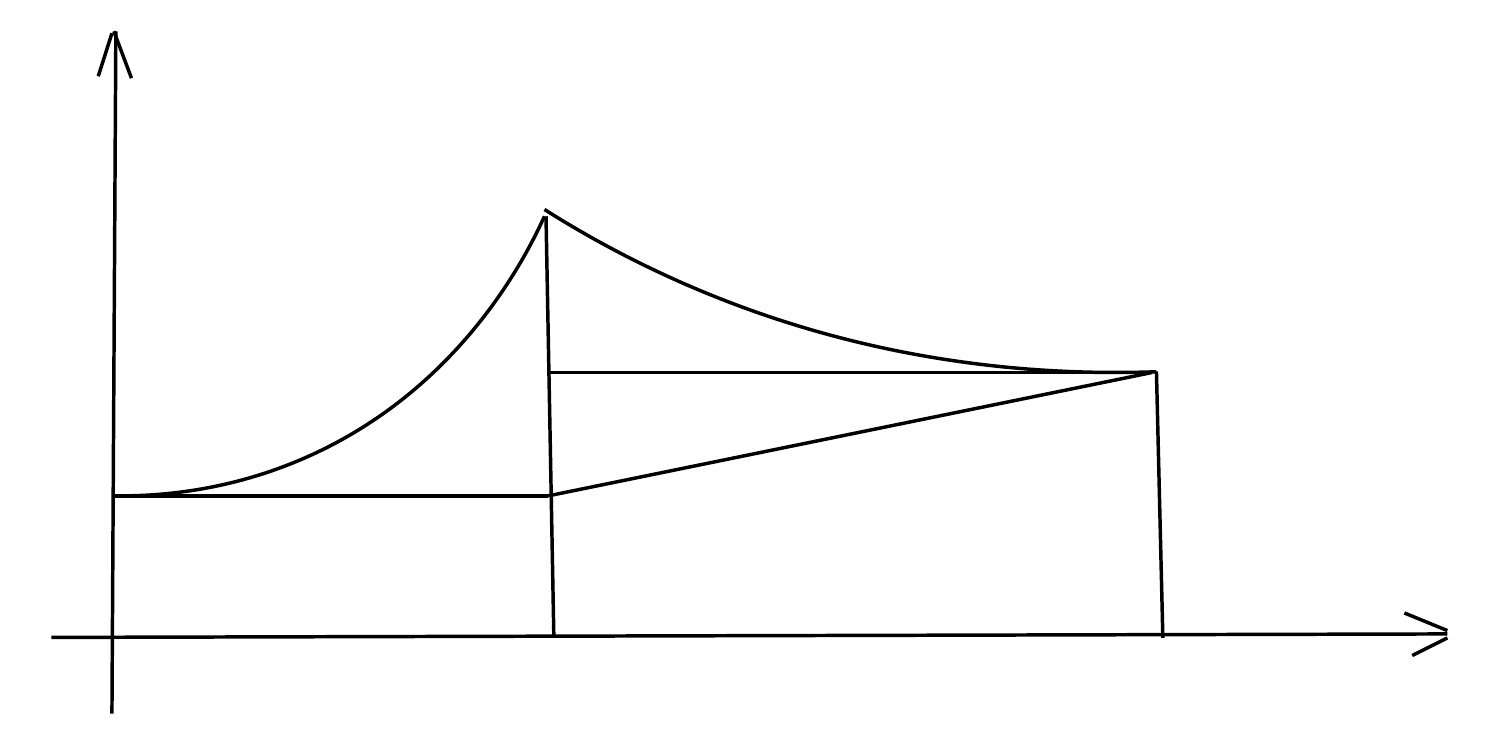,width=8.0cm,angle=0} }}
\vspace{-20pt}
\end{center}
\caption{The image $f(\Omega_{1,k})$.} 
\end{figure}

Let $t\in (0,\ell_X(o_{1,2}))$ be the foot of the orthogonal to the real axis from the intersection point $f(a_{i,k})\cap f(o_{2,3})$ which is one vertex of $f(\Omega_{i,k})$. Then $f(\Omega_{i,k})$ contains Euclidean rectangles $R_1:=[0,t]\times [0,\frac{1}{4}\ell_n]$ and $R_2:=[t,\ell_X(o_{1,2})]\times [0,\frac{1}{2}\ell_{n+1}]$ (see Figure 9). Notice that the second rectangle has a bigger height. Since the heights of both rectangles are bounded above, the restriction to $f^{-1}(R_1\cup R_2)$ of the differentiable function $f$ has a  derivative bounded between two positive constants for all $n$. Therefore the diffeomorphism $f$ is a quasiconformal map on $f^{-1}(R_1\cup R_2)$ with the quasiconformal constant bounded above independently in $n$.

We construct a foliation in $R_1$ by taking $v_1(x+iy)=y$. The leaves are horizontal lines and the Dirichlet integral is $\mathcal{D}_{R_1}(v_1)\asymp\frac{1}{4}\ell_n t\leq \frac{1}{4}\ell_n\log\frac{1}{\ell_n}.$ The transverse measure in $R_1$ is equal to $\frac{1}{4}\ell_n$. 

In order to continue the foliation of $R_1$ to a partial foliation on $R_2$ using Lemma \ref{lem:trapezoid}, we first need to prove that $\ell_X(o_{1,2})-t\geq \ell_{n+1}$. Consider the two hyperbolic quadrilateral obtained from the pentagon $\Omega_{1,k}$ by drawing an orthogonal to the side $o_{1,2}$ from the vertex $o_{2,3}\cap a_{1,k}$ (see Figure 8). 
Note that the length of the side $AC$ is $\ell_X(o_{1,2})-t$ and the length of the side $CB$ is $t$ (see Figure 8).  Using a trigonometric identity for quadrilaterals with three right angles (see \cite[Theorem 2.3.1]{Buser}), we get
$$
\cosh t=\tanh b\coth \frac{\ell_n}{4}\ \mathrm{and}\ \cosh (\ell_X(o_{1,2})-t)=\tanh b\coth \frac{\ell_{n+1}}{2}
$$
where $b$ is the length of the orthogonal arc to $o_{1,2}$ as in Figure 8.

By dividing the first equality with the second we get $\cosh t /[\cosh (\ell_X(o_{1,2})-t)]=\coth \frac{\ell_n}{4}/\coth\frac{\ell_{n+1}}{2}$
Since $1<\ell_n/\ell_{n+1}<2$, $\ell_n\to 0$ and $\ell_X(o_{1,2})\to\infty$ as $n\to\infty$, it follows that $\ell_X(o_{1,2})-t\geq 1/2\geq \ell_{n+1}$ for $n$ large enough and we can apply Lemma \ref{lem:trapezoid}.

Let $Q_2$ be the trapezoid in $R_2$ whose height is $[t,\ell_X(o_{1,2})]$ and the two bases are $\{ t\}\times [0,\frac{1}{4}\ell_n]$ and $\{\ell_X(o_{1,2})\}\times [0,\frac{1}{2}\ell_{n+1}]$. We choose a foliation $v_2:Q_2\to\mathbb{R}$ such that it connects points on $\{ t\}\times [0,\frac{1}{4}\ell_n]$ to the points on $\{\ell_X(o_{1,2})\}\times [0,\frac{1}{2}\ell_{n+1}]$ by Euclidean segments in a way that the correspondence between the two bases is linear with respect to the Euclidean distance. By Lemma \ref{lem:trapezoid}, the choice of $v_2$ can be made such that the transverse measure in $Q_2$ is equal to $\frac{1}{4}\ell_n$ and $\mathcal{D}_{Q_2}(v_2)\asymp \ell_n\log\frac{1}{\ell_n}.$ Note that the foliations $v_1$ and $v_2$ have the same transverse measures on $\{ t\}\times [0,\frac{1}{2}\ell_n]$ and they define a foliation on $R_1\cup R_2$ with trasverse measure $\frac{1}{4}\ell_n$ and Dirichlet integral of the order $\ell_n\log\frac{1}{\ell_n}.$ 

We perform this construction in all pentagons and for all pairs of pants. The leaves of the partial foliations in the pairs of pants glue in such a fashion that they are continued indefinitely and converge to infinity at both ends. Each leaf is homotopic to a single geodesic from $\{ g_k\}_k$. The transverse measures on the boundary geodesics of the pairs of pants do not match. We scale the foliation to match the transverse measures on the boundary geodesics of the pants decomposition. The scaling is done by multiplying the functions $v:R_1\cup Q_2\to\mathbb{R}$ by an appropriate positive number.

We start from level $1$ where the first two pairs of pants $P_1$ and $P_1'$ are glued along the boundary geodesic $\alpha_1$ of length $\ell_1=2$. The transverse measure on $\alpha_1$ from both $P_1$ and $P_1'$ equals $\ell_1$. The trasverse measures on the other boundaries of $P_1$ and $P_1'$ are equal to $\frac{1}{2}\ell_1$. We multiply the foliation functions $v_1:P_1\to\mathbb{R}$ and $v_1':P_1'\to\mathbb{R}$ by $\frac{1}{\ell_1}$ and note that their Dirichlet integral is multiplied by $\frac{1}{\ell_1^2}$. Therefore the new partial foliation on $P_1\cup P_1'$ has transverse measure $1$ on $\alpha_1$ and equal transverse measures $\frac{1}{2}$ on the four boundary geodesics. The Dirichlet integral of the new partial foliation on $P_1\cup P_1'$ is finite.

Assume that we have defined a new partial foliation on the union of the pairs of pants of the level $n-1$ such that the transverse measure on each of the $2^n$ boundary geodesics is equal to $\frac{1}{2^n}$. On each pair of pants of the level $n$, which is attached along a boundary geodesic $\alpha_{n-1}^j$ to the level $n-1$, we have a partial foliation with transverse measure equal to $\ell_{n-1}$ on $\alpha_{n-1}^j$ and equal to $\ell_{n-1}/2$ on the other two pairs of pants. We scale the partial foliation on the pair of pants by $\frac{1}{2^n \ell_{n-1}}$ to obtain partial foliation whose transverse measure is $\frac{1}{2^n}$ on $\alpha_{n-1}^j$ and $\frac{1}{2^{n+1}}$ on the other two boundaries of the pair of pants on the level $n$. Continuing in this fashion, the partial foliations on the pairs of pants glue to a global partial foliation $\mathcal{F}$ of the Riemann surface $X$ whose each leaf leaves every compact subset of $X$ in both directions and such that the transverse measure of the boundary geodesic between level $n-1$ and level $n$ pairs of pants is $\frac{1}{2^n}$. Note that the transverse measures on the cuffs are proportional to the hyperbolic lengths, which allowed us to claim that the measures match as long as they have the same total mass.

It remains to compute the Dirichlet integral of $\mathcal{F}$. On the pair of pants on the level $n$, the Dirichlet integral is of the order $(\frac{1}{2^n\ell_n})^2\ell_n\log\frac{1}{\ell_n}$. By summing over all $2^n$ pairs of pants at the level $n$ and then over all $n$, we obtain that the Dirichlet integral of $\mathcal{F}$ over $X$ is of the order
$$
D_{P_1\cup P_1'}(\mathcal{F})+\sum_{n=2}^{\infty} 2^n(\frac{1}{2^n\ell_n})^2\ell_n\log\frac{1}{\ell_n}=D_{P_1\cup P_1'}(\mathcal{F})+\sum_{n=2}^{\infty}\frac{1}{2^n\ell_n}\log\frac{1}{\ell_n}.
$$
Since $\ell_n=\frac{(n+1)^r}{2^{n+1}}$, the Dirichlet integral of $\mathcal{F}$ is of the order
$$
\sum_{n=2}^{\infty} \frac{1}{(n+1)^r}((n+1)\log 2-r\log (n+1))\leq\sum_{n=2}^{\infty}\frac{1}{n^{r-1}}.
$$
By $r>2$ we conclude that $\mathcal{D}_X(\mathcal{F})<\infty$. Therefore by Theorem \ref{thm:main}, there exists a non-trivial integrable holomorphic quadratic differential $\varphi$ on $X$ whose all horizontal trajectories are cross-cuts. By Theorem \ref{thm:int-non-par-cross-cut}, the surface $X$ is not parabolic.
\end{proof}

\author{Dragomir \v Sari\' c, Department of Mathematics,  Graduate Center and Queens College, CUNY, Dragomir.Saric@qc.cuny.edu}

\end{document}